\documentclass{article}
\usepackage[utf8]{inputenc}
\usepackage{hyperref}
\usepackage{todonotes}
\usepackage{amsmath}
\usepackage{amsthm}
\usepackage{amsfonts}
\usepackage{mathtools}
\usepackage[ruled, linesnumbered, vlined]{algorithm2e} 
\usepackage{enumitem}
\usepackage{tikz}
\usetikzlibrary[shapes.geometric, calc]
\usetikzlibrary{arrows.meta,arrows}

\usepackage[a4paper,top=1cm,bottom=3cm,left=2cm,right=2cm,marginparwidth=4.5cm]{geometry}

\usepackage{bm}

\usepackage{cleveref}

\newcommand{\leo}[1]{{#1}}
\newcommand{\leonew}[1]{#1}
\newcommand{\yuki}[1]{#1}
\newcommand{\yukinew}[1]{#1}
\newcommand{\yukinewnew}[1]{{#1}}
\newcommand{\markj}[1]{#1}

\newtheorem{theorem}{Theorem}
\newtheorem{lemma}{Lemma}
\newtheorem{corollary}{Corollary}

\newtheorem{observation}{Observation}

\theoremstyle{definition}
\newtheorem{definition}{Definition}

\usepackage{enumitem}
\setlist[description]{leftmargin=\parindent,labelindent=\parindent}

\newcommand{\U}[1][]{
\ifthenelse{\equal{#1}{}}{\mathcal{U}}{\mathcal{U}_{#1}}%
}
\newcommand{\N}[1][]{
\ifthenelse{\equal{#1}{}}{\mathcal{N}}{\mathcal{N}_{#1}}%
}
\newcommand{\NP}[1][]{
\ifthenelse{\equal{#1}{}}{\mathcal{N}^{()}}{\mathcal{N}^{()}_{#1}}%
}
\newcommand{\Ui}[1][]{
\ifthenelse{\equal{#1}{}}{\dot{\mathcal{U}}}{\dot{\mathcal{U}}_{#1}}%
}
\newcommand{\Ni}[1][]{
\ifthenelse{\equal{#1}{}}{\mathcal{N}^l}{\mathcal{N}^l_{#1}}%
}

\newcommand{\NiP}[1][]{
\ifthenelse{\equal{#1}{}}{\dot{\mathcal{N}}^{()}}{\dot{\mathcal{N}}^{()}_{#1}}%
}

\newcommand{\RR}{\mathbb{R}}

\newcommand{\move}[1]{\mathrel{\raisebox{-2pt}{$\xrightarrow{#1}$}}}
\newcommand{\isom}[1][]{
\ifthenelse{\equal{#1}{}}{\simeq}{\mathrel{\raisebox{-2pt}{$\overset{#1}{\simeq}$}}}%
}

\DeclarePairedDelimiter{\ceil}{\lceil}{\rceil}

\DeclareMathOperator{\diam}{diam}

\DeclareMathOperator{\Orch}{Orch}

\newcommand{\rNNI}{\mathrm{rNNI}}

\title{Orchard Networks are Trees with Additional Horizontal Arcs}
\author{Leo van Iersel \footnote{Delft Institute of Applied Mathematics, Delft University of Technology, Mekelweg 4, 2628 CD, Delft, The Netherlands, \{L.J.J.vanIersel, R.Janssen-2, M.E.L.Jones, Y.Murakami\}@tudelft.nl. Research funded in part by the Netherlands Organization for Scientific Research (NWO), including Vidi grant 639.072.602, and partly by the 4TU Applied Mathematics Institute. Mark Jones was also supported by the KLEIN grant OCENW.KLEIN.125.} \and Remie Janssen\footnotemark[1] \and Mark Jones\footnotemark[1] \and Yukihiro Murakami\footnotemark[1]
      }
\date{\today}

\begin{document}

\maketitle
\begin{abstract}
Phylogenetic networks are used in biology to represent evolutionary histories. The class of orchard phylogenetic networks was recently introduced for their computational benefits, without any biological justification. Here, we show that orchard networks can be interpreted as trees with additional \emph{horizontal} arcs. \leo{Therefore, they are closely related to tree-based networks, where the difference is that in tree-based networks the additional arcs do not need to be horizontal.} Then, we use this new characterization to show that the space of orchard networks is connected under the rNNI rearrangement move, with a diameter of at most $4kn+n\ceil{\log_2(n)}+2k+6n-8$.
\end{abstract}




\maketitle

\section{Introduction}
Phylogenetic trees and networks are used in biology to represent evolutionary histories. Trees primarily show the vertical evolutionary processes, whereas networks are used to augment this view with the inclusion of horizontal, or reticulate, processes such as horizontal gene transfer, hybridization, and recombination \cite{bapteste2013networks,elworth2019advances,blais2021past}.

Sometimes, only a subset of all phylogenetic networks is considered, either to ease the computation of a good representation of the actual evolutionary history (e.g., \cite{van2014trinets,bordewich2018constructing,erdHos2019class,van2019practical,borst2020new}), or because the evolutionary history is expected to have a certain structure (e.g., \cite{markin2019robinson}). 
Examples of such subsets are the classes of phylogenetic networks called ``tree-child'' networks, ``tree-based'' networks, and ``temporal'' networks. Tree-based networks are mainly used for their biological interpretation as a tree with additional arcs (e.g., \cite{francis2015phylogenetic,cardona2015reconstruction,fischer2020classes}). \leo{Temporal and tree-child networks are especially interesting because of their computation benefits and nice mathematical properties.}
Tree-child networks can be biologically interpreted as networks in which each reticulate event leaves a trace in the genetic information of the studied species \cite{cardona2008comparison,markin2019robinson}. \leo{Networks with a temporal labelling} can be biologically interpreted as networks where each reticulation represents a hybrid speciation event~\cite{moret2004phylogenetic,baroni2006hybrids}. 
\leo{Networks that are tree-child and have a temporal labelling are sometimes simply called \emph{temporal}~\cite{humphries2013complexity}.} \leonew{It should be noted though that extinctions or undersampling may lead to phylogenetic networks that are not temporal or not tree-based.}

Another class of networks is the class of \emph{orchard networks}, which were recently introduced for their computational benefits~\cite{janssen2021cherry,erdHos2019class} \leonew{while being more general than tree-child networks}. They are usually defined as networks that can be reduced by a series of cherry-picking actions \leo{(for this reason, they were called ``cherry picking networks'' in~\cite{janssen2021cherry})}. This comes up naturally when one builds networks by repeatedly considering pairs of taxa that seem to be closely related in the data \cite{bordewich2018constructing,van2019practical,erdHos2019class,bai2021defining}. Recently, it was shown that they can be characterized more structurally using acyclic cherry covers~\cite{van2021unifying}.

\leonew{In this paper, we show that orchard networks can also be characterized in a very natural way. We prove that orchard networks are precisely those phylogenetic networks that can be obtained from a phylogenetic tree by inserting horizontal arcs, see \Cref{fig:RealHGT} for an example. Hence, they are closely related to LGT networks~\cite{pons2019generation} except that orchard networks do not (necessarily) specify which arcs are LGT-arcs. Our characterization of orchard networks is similar to the \markj{previously mentioned} concept of \emph{tree-based} networks, which are those networks that can be obtained from a phylogenetic tree, called a \emph{base tree}, by inserting arcs, called \emph{linking arcs}. The difference is that for orchard networks the linking arcs need to be horizontal while for tree-based networks they do not.}

\leonew{Our characterization of orchard networks can be seen as a time-consistency property, which we call an ``HGT-consistent labelling''. It basically says that orchard networks are consistent with an evolutionary history in time in which reticulate events represent instantaneous (horizontal) transfers such as LGT events. This is similar, but not the same, as the notion of time-consistency, or temporal labelling, that is commonly found in the phylogenetic networks literature~\cite{moret2004phylogenetic,baroni2006hybrids}. The difference is that in a temporal labelling \emph{both} arcs entering a reticulation need to be horizontal, which is more natural when reticulations represent hybridization events. This notion has been widely popular in defining network classes, which have been explored in relation to metrics~\cite{cardona2008distance}, encodings~\cite{cardona2008distance}, and reconstruction methods~\cite{borst2020new}. Another related notion of time-consistency was introduced in~\cite{erdem2005temporal}, which allows bidirectional horizontal arcs as well.} 

Unlike the other definitions of orchard networks, the \leonew{``HGT-consistent labelling''} that we introduce here can easily be interpreted biologically. This makes the class of orchard networks even more relevant, as it has a natural biological interpretation as a tree with horizontal arcs. \leonew{Nevertheless, as with the classes of temporal and tree-based networks, extinctions or undersampling may lead to networks that are not orchard, because one or more LGT arcs need to be drawn forward in time.}

\begin{figure}
    \centering
    \includegraphics[width=0.5\textwidth]{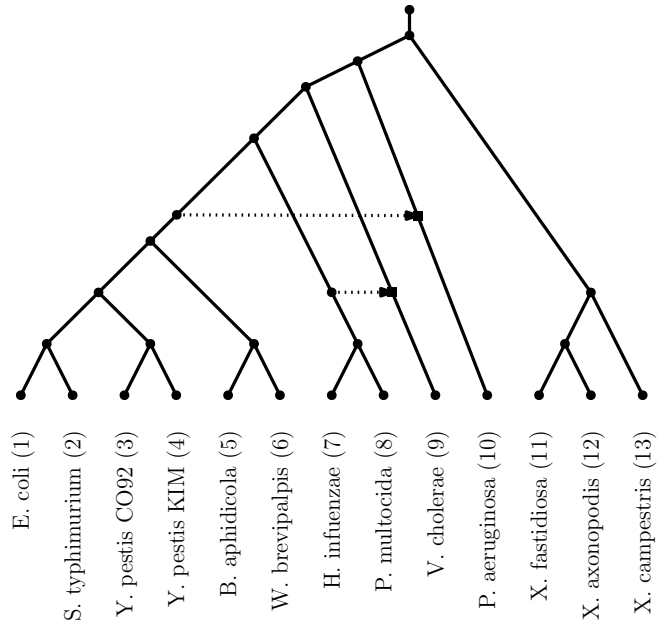}
    \caption{A network on taxa within the $\gamma$-Proteobacteria class with two reticulate events~\cite{nakhleh2005riata}. 
    The dotted arcs with the arrowheads illustrate the passing of genetic material via a horizontal transfer.
    \markj{In this and all subsequent figures, the non-horizontal arcs without arrowheads are directed downwards.}
    The network is a binary orchard network, since it can be reduced by a sequence of cherry reductions (see \Cref{subsec:Orch} for formal definitions of \leonew{such reductions}). \leonew{One possible} sequence that would reduce the network is
    $(1,2)(3,4)(2,4)(5,6)(4,6)(7,8)(9,8)(10,6)(6,8)(8,9)(9,10)(11,12)(12,13)(10,13)$.}
    \label{fig:RealHGT}
\end{figure}

To showcase the mathematical utility of the new characterization, we use it to prove that the space of orchard networks is connected under rNNI moves. 
Connectedness of search spaces under rearrangement moves, such as the rNNI move, is important because of their use in local search heuristics and Bayesian methods in phylogenetics (e.g., \cite{markin2019robinson,bouckaert2019beast,wen2018inferring}). Although it has already been proven that rearrangement moves are sufficient to connect several classes of networks \cite{bordewich2017lost,klawitter2020spaces,erdHos2021rooted}, the connectedness of the space of orchard networks has not been investigated yet. 

In this paper, we fill this gap in the literature, which may have resulted from the late introduction of this relatively new class of networks. Because orchard networks naturally occur as results of statistical models for network generation, such as the level-$k$ LGT networks from \cite{pons2019generation}, this result may prove to be especially important. Indeed, our results then show that it is possible to use such a model as a prior in a Bayesian method in conjunction with rNNI or rSPR moves, while keeping the search space connected.

\section{Preliminaries}
In this section, we define phylogenetic networks and \yukinewnew{recall} the 
\leonew{original definition} of orchard networks. Then we recall the definition of an rNNI move on a phylogenetic network, and some results related to these moves.

\begin{definition}
A \emph{directed phylogenetic network} on a set of \emph{taxa} $X$ is a directed acyclic graph whose nodes are of the following types:
\begin{description}
\item[Root] a node with indegree-0 and outdegree-1.
\item[Tree node] a node with indegree-1 and outdegree at least 2.
\item[Reticulation] a node with indegree at least 2 and outdegree-1.
\item[Leaf] a node with indegree-1 and outdegree-0.
\end{description}
The leaves are bijectively labelled by the taxa in~$X$. Because of this, we will refer to a leaf~$l$ by its label in~$X$ and vice versa. \leonew{Non-leaf nodes will also be called \emph{internal} nodes. Furthermore, we will refer to directed phylogenetic networks as \emph{networks} for short.}
\end{definition}

A network where all tree nodes and reticulations have total degree (sum of indegree and outdegree) exactly 3 is called a \emph{binary} network.
We use the term \emph{non-binary} to mean \emph{not necessarily binary}; in particular, binary networks are also non-binary.
Most networks we will consider in this paper are directed binary phylogenetic networks.
We shall prove results on non-binary networks only in \Cref{sec:NonBinary}.

As a network $N$ is a directed acyclic graph, we have a natural ordering on the nodes of $N$. If $N$ contains an arc $(u,v)$, then we say that $u$ is a \emph{parent} of $v$, that~$v$ is a \emph{child} of $u$, \leonew{that~$u$ is the \emph{tail} of~$(u,v)$ and that~$v$ is the \emph{head} of~$(u,v)$}. If there is a directed path from $u$ to $v$, then we say that $u$ is \leonew{\emph{above}} $v$ and $v$ is \leonew{\emph{below}} $u$.
An arc~$(u,v)$ is a \emph{reticulation arc} if~$v$ is a reticulation; it is a \emph{tree arc} otherwise.

\yukinewnew{The \emph{reticulation number} $r(N)$ of a network $N$ is the total number of reticulation arcs minus the total number of reticulation nodes in~$N$.
In a binary network, the reticulation number is simply the number of reticulation nodes.}

We say that two \leo{networks~$N,N'$ are \emph{isomorphic} if}
there exists a bijection~$f$ between the nodes of~$N$ and the nodes of~$N'$
such that~$(u,v)$ is an arc of~$N$ if and only if~$(f(u),f(v))$ is an arc of~$N'$ and each labelled node of~$N$ is mapped to a node in~$N'$ with the same label.

Let~$N$ be a network and let~$v$ be a node in~$N$.
For an arc~$(u,v)$ in~$N$, \emph{contracting the arc~$(u,v)$} is the action of deleting the arc~$(u,v)$ and identifying~$u$ and~$v$.
We say that a network~$N'$ is a \emph{binary resolution} of~$N$ if~$N'$ is binary and a network isomorphic to~$N$ can be obtained from~$N'$ by contracting arcs.

\subsection{Orchard networks}\label{subsec:Orch}
Orchard networks were first introduced as networks that can be reduced by picking \emph{cherries} and \emph{reticulated cherries} \cite{janssen2021cherry,erdHos2019class}. In this section, we recall this definition of orchard networks.

\begin{definition}
An ordered pair of leaves $(x,y)$ in a network $N$ is a \emph{cherry} if $x$ and $y$ share a common parent. The pair $(x,y)$ is a \emph{reticulated cherry} if the parent $p_x$ of $x$ is a reticulation, and $p_x$ and $y$ share a common parent. If $(x,y)$ is a cherry or a reticulated cherry, we call it a \emph{reducible pair}.
\end{definition}

\yukinewnew{\emph{Suppressing} a node~$v$ with exactly one parent~$u$ and exactly one child~$w$ refers to deleting the node~$v$ and adding an arc~$(u,w)$.}

\begin{definition}\label{def:reduction}
Let $N$ be a non-binary network and let $(x,y)$ be a pair of leaves in~$N$. Let~$p_x,p_y$ denote the parents of~$x,y$, respectively, and let~$g_x$ denote the parent of~$p_x$ that is not~$p_y$, and let~$g_y$ denote a parent of~$p_y$. \emph{Reducing} (or \emph{picking}) the pair $(x,y)$ in $N$ consists of the following:
\begin{itemize}
    \item If $(x,y)$ is a cherry, then delete~$x$, and suppress~$p_x$ if~$p_x$ is consequently a node of indegree-1 and outdegree-1.
    \item If $(x,y)$ is a reticulated cherry, delete the arc~$(p_y,p_x)$, and suppress indegree-1 and outdegree-1 nodes ($p_x$ \leo{and} $p_y$ are the only candidates for this suppression).
    \item If~$(x,y)$ is not a cherry nor a reticulated cherry, then do nothing.
\end{itemize}
\end{definition}

It is easy to see that the graph obtained by reducing a pair from a network is still a network, precisely because we suppress indegree-1 and outdegree-1 nodes.
We denote the network obtained by reducing a pair~$(x,y)$ from a network~$N$ by~$N(x,y)$.
Let~$S$ be a sequence of ordered pairs on distinct elements. The network obtained by repeatedly reducing the pairs of~$S$ in order from~$N$ is denoted~$NS$.

\begin{definition}
A network is \emph{orchard} if there exists a sequence of ordered pairs~$S$ such that~$NS$ is a tree with one leaf. In such a case, we say that~$S$ reduces~$N$.
\end{definition}

We call a sequence of ordered pairs a \emph{cherry-picking sequence} if it reduces some orchard network, such that the length of the sequence is minimal.
It was shown independently in \cite{janssen2021cherry, erdHos2019class} that 
\markj{picking any reducible pair in an orchard network results in another orchard network}
(also for non-binary networks \cite{janssen2021cherry}).
This means that in general, orchard networks may have \leonew{multiple} cherry-picking sequences that reduce them.
See \Cref{fig:RealHGT} for an example of an orchard network.

\subsection{Tree-based networks}

Tree-based networks were introduced as those that can be obtained by adding arcs between arcs of a given \emph{base tree}~\cite{francis2015phylogenetic}. 
We recall the following definition.

\begin{definition}
A binary network~$N$ is \emph{tree-based} with \emph{base tree}~$T$ if~$N$ can be obtained from~$T$ in the following steps:
\begin{enumerate}
    \item Replace some arcs of~$T$ by paths, whose internal nodes are called \emph{attachment points}; each attachment point is of indegree-1 and outdegree-1;
    \item Place arcs between attachment points, called \emph{linking arcs}, so that the graph contains no nodes of total degree greater than~$3$, and so that it remains acyclic; and
    \item Suppress all attachment points not incident to any linking arcs.
\end{enumerate}
\end{definition}

\leonew{Definitions regarding rearrangement moves are given in Section~\ref{sec:moves}.}

\section{\leonew{Characterization} of orchard networks}
In this section, we prove that orchard networks can be characterized as phylogenetic networks that admit a certain type of time-consistent labelling.
We show this first for binary orchard networks, and later extend the characterization to non-binary orchard networks.

\subsection{Binary orchard networks}

\begin{definition}\label{def:HGT}
Let $N$ be a binary phylogenetic network \leonew{with node set~$V$}. An \emph{HGT-consistent labelling} of $N$ is a labelling $t:V\rightarrow \RR$ such that 
\begin{enumerate}
    \item For all arcs $(u,v)$, $t(u)\leq t(v)$ and equality is only allowed if $v$ is a reticulation.
    \item For each internal node $u$, there is a child $v$ of $u$ such that $t(u)<t(v)$.
    \item For each reticulation $r$ with parents $u$ and $v$, exactly one of $t(u)=t(r)$ and $t(v)=t(r)$ holds.
\end{enumerate}
\end{definition}

The properties listed above in Definition~\ref{def:HGT} will be referred to as Properties 1,2, and 3, respectively.
The biological interpretation of a network with an HGT-consistent labelling is that reticulations are caused within the network by horizontal transfers.
Property 3 ensures that reticulation nodes are contemporaneous with one of their parents -- in particular, with the parent from which genetic material is passed via the horizontal arc.
Intuitively, one can view a network with an HGT-consistent labelling as a tree with horizontal arcs added to it.
This means that a base tree of a given network with HGT-consistent labelling can be obtained by deleting all reticulate arcs with endpoints of the same label (horizontal arcs). 
The following observation follows immediately.

\medskip



\begin{observation}\label{obs:HGT-->TB}
Each binary network that admits an HGT-consistent labelling is tree-based.
In particular, a base tree can be obtained by deleting all reticulation arcs where the tail and head have the same time labels.
\end{observation}

The converse of \Cref{obs:HGT-->TB} does not hold. \leonew{This can be seen as follows.}
We say that a network \emph{contains a crown} if there exists a set of nodes $\{u_1,\ldots, u_k, v_1,\ldots,v_k\}$ with edges~$\{\leonew{(u_i,v_i), (u_i,v_{i+1})} : i\in [k]\}$ where~$[k]=\{1,\ldots,k\}$, where the indices are taken modulo~$k$. Consider a tree-based network that contains a crown, \leonew{such as the network in~\Cref{fig:crown}. Such a network does not admit an HGT-consistent labelling. To see this, first note that in such a labelling~$t$, without loss of generality, $t(u_1)=t(v_1)$ (by Property~3). Moreover, if \markj{$t(u_i)=t(v_i)$} then it follows that \markj{$t(u_i)<t(v_{i+1})$} (by Property~2) and hence that~\markj{$t(u_{i+1})=t(v_{i+1})$} (by Property~3). Then, by induction, it follows that~\markj{$t(u_1)<t(u_2) <\dots < t(u_k) < t(u_1)$}, a contradiction.}

We now show that every binary orchard network admits an HGT-consistent labelling of a certain form.

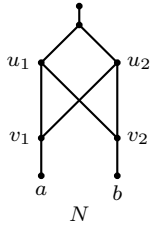
\begin{figure}
    \centering
    \begin{tikzpicture}
	 \tikzset{edge/.style={thick}}
     \tikzset{arc/.style={-Latex,thick}}
     \tikzstyle{every node}=[font=\footnotesize]
	 \begin{scope}[xshift=0cm,yshift=0cm,xscale=.5,yscale=.5]
	\draw[thick, fill, radius=0.06] (0,.5) circle;
	\draw[thick, fill, radius=0.06] (0,0) circle;
	\draw[edge] (0,.5) -- (0,0);
	\draw[edge] (0,0) -- (-1,-1);
	\draw[edge] (0,0) -- (1,-1);
	\draw[thick, fill, radius=0.06] (-1,-1) circle node[left] {$u_1$};
	\draw[thick, fill, radius=0.06] (-1,-3) circle node[left] {$v_1$};
	\draw[thick, fill, radius=0.06] (-1,-4) circle node[below] {$a$};
	\draw[thick, fill, radius=0.06] (1,-1) circle node[right] {$u_2$};
	\draw[thick, fill, radius=0.06] (1,-3) circle node[right] {$v_2$};
	\draw[thick, fill, radius=0.06] (1,-4) circle node[below] {$b$};
	\draw[edge] (-1,-1) -- (1,-3);
	\draw[edge] (-1,-1) -- (-1,-3);
	\draw[edge] (1,-1) -- (1,-3);
	\draw[edge] (1,-1) -- (-1,-3);
	\draw[edge] (-1,-3) -- (-1,-4);
	\draw[edge] (1,-3) -- (1,-4);
	\draw (0,-5) node {$N$};
	\end{scope}
	\end{tikzpicture}
    \caption{A binary tree-based network~$N$ on the taxa set~$\{a,b\}$ \leonew{that does not admit an HGT-consistent labelling since it contains} a crown on the nodes~\leonew{$u_1,u_2,v_1,v_2$.}}
    \label{fig:crown}
\end{figure}

\begin{lemma}\label{lem:OrchardThenHGTConsistent}
Each binary orchard network admits an HGT-consistent labelling $t$ where
two nodes have the same label \leonew{only if they are} a parent-child pair where the parent is a tree node and the child a reticulation.
\end{lemma}
\begin{proof}
Let $N$ be a binary orchard network \markj{with leaves $l_1, \dots, l_n$}, and let $S=(x_1,y_1),\ldots,(x_m,y_m)$ be a cherry-picking sequence for $N$. Because $N$ is binary, it can be reconstructed from $S$ by starting with the one-leaf tree with leaf $y_m$ and reattaching the pairs from $S$ in reverse order\footnote{We refer the interested reader to~\cite{janssen2021cherry} for more information on this construction.}. Now label the root node $\rho$ with $t(\rho)=0$, each leaf \markj{$l_j$ with $t(l)=m+j$}, and each \leonew{internal} node $v$ added when reattaching the pair $(x_i,y_i)$ with $t(v)=m+1-i$. 

We show that~$t$ is indeed an HGT-consistent labelling of~$N$.
When adding a pair to a network, 
\markj{the two newly introduced nodes \leonew{are not} above any existing \leonew{internal} nodes, and have a greater time labelling than any other existing \leonew{internal} nodes. Thus any \leonew{internal} node has time label greater than or equal to that of any of its parents.}
Adding to the fact that we label the leaves so that they have labels of at least~$m+1$ and internal nodes have labels of at most~$m$, we have that for all arcs~$(u,v)$,~$t(u)\le t(v)$.
The labelling of the leaves also means that two nodes in the network have the same label under~$t$ \markj{only} if they are a parent-child pair where the parent is a tree node and the child is a reticulation. 
Therefore, Property 1 of the HGT-consistent labelling is satisfied.
To see that~$t$ satisfies Property 2, we look at tree nodes and reticulations separately.
A tree node~$u$ has two children, one of which is possibly added to the network at the same time as~$u$. The other child~$v$ of~$u$ is either a leaf or an internal node that is added to the network after~$u$ has been added. But this would mean that~$t(u)<t(v)$.
On the other hand, a reticulation~$r$ has one child~$c$. Every reticulation node is added to the network with one of its \markj{(non-reticulation)} parents; this means that~$c$ is either a leaf or an internal node that is added to the network after~$r$ has been added. This implies~$t(r) < t(c)$.
So Property 2 is satisfied.
Finally, to see that Property 3 is also satisfied, consider a reticulation~$r$ with parents~$u$ and~$v$. The reticulation~$r$ must have been added to the network with one of its parents, say~$u$, so that~$t(u) = t(r)$. This means that the node~$v$ was already in the network when~$r$ was added; due to how we have defined~$t$, we must have that~$t(v) < t(r)$.
Thus~$t$ satisfies Property 3, and therefore it is an HGT-consistent labelling.

In addition, this gives a labelling $t$ \leonew{in which each label is used at most twice.}
Indeed for each $i\in\{1,\ldots,m\}$, the nodes with label $i$ are added to the network when $(x_i,y_i)$ is reattached, and for each such reattachment, at most two nodes are added to the network.
Observe that under this construction, if two nodes have the same label, then they must be a parent-child pair where the parent is a tree node and the child a reticulation.
\end{proof}

\begin{lemma}\label{lem:HGT-->Orchard}
Each binary network that admits an HGT-consistent labelling is orchard.
\end{lemma}
\begin{proof}
Suppose a binary network $N$ admits an HGT-consistent labelling $t$. We will prove that $N$ is orchard by proving that each network that admits an HGT-consistent labelling which has at least one internal node must contain a cherry or reticulated cherry. Moreover, after reducing such a pair, the resulting network still admits an HGT-consistent labelling. Therefore, any HGT-consistent network can be reduced to a tree with one leaf, and is thus orchard.

Let $x$ be an internal node with $t(x)$ maximal. If $x$ is a reticulation, then its child $l$ must be a leaf, and one of its parents $p$ has label $t(p)=t(x)$, by Property~3. 
The node~$p$ cannot be a reticulation as this would contradict Property 2 of HGT-consistent labellings.
Hence, the other child $l'$ of $p$ must be a leaf as well, and 
the reducible pair $(l,l')$ is a reticulated cherry in $N$. 
Reducing this reticulated cherry, we obtain a new network $N'$, which 
still admits an HGT-consistent labelling $t'=t\vert_{V(N')}$. 
If $x$ is a tree node, then either both of its children are reticulations, one of its children $v$ is a reticulation node, or both its children are leaves.
In the first case, the reticulation children~$v_1$ and~$v_2$ 
must have labels~$t(v_1) = t(v_2) = t(x)$, since~$x$ has the greatest time label 
out of internal nodes and 
\markj{internal nodes have time label greater than or equal to that of their parents, by Property 1.}
But this contradicts Property 2, so this case is not possible.
In the second case, the reticulation child has label $t(v)=t(x)$, and we can reduce the reticulated cherry involving $x$ and $v$ as in the previous case. In the third case, $x$ has two leaf children, which must thus form a cherry. After reducing this cherry, the network still admits an HGT-consistent labelling, which can be obtained by restricting $t$ to the remaining nodes. 
\end{proof}

A direct consequence of these lemmas is the following new characterization of orchard networks as networks that admit an HGT-consistent labelling.

\begin{theorem}\label{thm:OrchIFFHGT}
A binary network $N$ is orchard if and only if it admits an HGT-consistent labelling.
\end{theorem}

The next corollary, which was also shown in \cite{huber2019rooting,van2021unifying}, follows from \Cref{obs:HGT-->TB} and \Cref{thm:OrchIFFHGT}.

\begin{corollary}
The class of binary orchard networks are contained in the class of binary tree-based networks.
\end{corollary}

This means in particular that orchard networks have a base tree.

\subsection{Non-binary orchard networks}\label{sec:NonBinary}

By recalling a key lemma from~\cite{van2021unifying} regarding non-binary orchard networks and their binary resolutions, we extend the HGT-consistent labelling characteristics to non-binary orchard networks.

\begin{lemma}[Lemma 11 of \cite{van2021unifying}]\label{lem:OrchardResolution}
A non-binary network~$N$ is orchard if and only if some binary resolution of~$N$ is orchard.
\end{lemma}

\begin{figure}
    \centering
    \begin{tikzpicture}
	 \tikzset{edge/.style={thick}}
     \tikzset{arc/.style={-Latex,thick}}
     \tikzstyle{every node}=[font=\footnotesize]
	 \begin{scope}[xshift=0cm,yshift=0cm,xscale=.5,yscale=.5]
	\draw[thick, fill, radius=0.06] (0,.5) circle;
	\draw[thick, fill, radius=0.06] (0,0) circle node[left] {$t_1$};;
	\draw[edge] (0,.5) -- (0,0);
	\draw[thick, fill, radius=0.06] (-1,-2) circle node[left] {$t_2$};
	\draw[thick, fill, radius=0.06] (-2,-4) circle node[left] {$t_3$};
	\draw[thick, fill, radius=0.06] (2.5,-5) circle node[right] {$r$};
	\draw[thick, fill, radius=0.06] (-3,-6) circle node[below] {$a$};
	\draw[thick, fill, radius=0.06] (3,-6) circle node[below] {$b$};
	\draw[edge] (0,0) -- (-3,-6);
	\draw[edge] (0,0) -- (3,-6);
	\draw[edge] (-1,-2) -- (2.5,-5);
	\draw[edge] (-2,-4) -- (2.5,-5);
	\end{scope}
	\end{tikzpicture}
    \caption{
    A non-binary orchard network on~$\{a,b\}$ that does not admit a time-labelling on the nodes that adheres to the `reticulations have the same labels as all but one of its parents' rule.
    Indeed, under this rule, exactly two of~$t_1,t_2,t_3$ must have the same time-label, \leonew{but no} labelling can satisfy this rule.
    Therefore, this labelling rule does not fully characterize the class of non-binary orchard networks.
    }
    \label{fig:non-orch}
\end{figure}
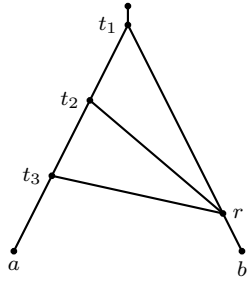

\begin{theorem}\label{thm:NonBinOrch}
A non-binary network~$N$ is orchard if and only if some binary resolution of~$N$ admits an HGT-consistent labelling.
\end{theorem}
\begin{proof}
The combination of \Cref{thm:OrchIFFHGT} and \Cref{lem:OrchardResolution} immediately gives the claim.
\end{proof}

Ideally, we \leonew{would like to} extend this characterization by finding a direct time labelling of non-binary networks that captures orchard networks, with a meaningful biological interpretation.
A natural generalization would be to consider a labelling where every reticulation is contemporaneous with all but one parent in the network. 
This implies the existence of a base tree, for which all arcs that are not in the base tree are horizontal arcs.
Unfortunately, this characterization does not fully capture the class of non-binary orchard networks.
\Cref{fig:non-orch} gives an example of a non-binary orchard network that does not admit a labelling under this property.

\section{Orchard network space}
We prove \leo{that} the space of binary orchard networks is connected using a strategy that is reminiscent of the proofs of connectedness in~\cite{janssen2021heading} for local head moves and in~\cite{erdHos2021rooted} for tree-based networks. We first move all reticulations to the top of the network
and then change the pendant trees below these reticulations. Before we prove connectedness, we will first \leonew{give the formal definitions and} introduce some structures and subgraphs we use in the proofs.

\subsection{Rearrangement moves}\label{sec:moves}
\leonew{We start by defining rSPR and rNNI moves on phylogenetic networks. Intuitively, rSPR moves can be seen as moving either the head or the tail of an arc and rNNI moves are local rSPR moves.}
\begin{definition}\label{def:rSPRrNNI}
(See \Cref{fig:rSPRmove})
Let $N$ be a binary network with \leonew{an arc~$(z,w)$ and an arc~$e$ with endpoints~$x$ and~$y$ (either $e=(x,y)$ or $e=(y,x)$), and let~$p$ and~$c$ be, respectively,} the parent and child of~$x$ that are not~\leonew{$y$}. The \emph{rSPR move} $(p,x,c)\move{e}(z,w)$ consists of the following: replace the arc\leonew{s} $(p,x)$, $(x,c)$, and $(z,w)$ with the arcs $(p,c)$, $(z,x)$, and $(x,w)$. \leonew{If $\{p,c\}\cap\{z,w\}\neq\emptyset$, then the move is an \emph{rNNI move}.}
\end{definition}

\begin{figure}[h]
    \centering
    \begin{tikzpicture}
	 \tikzset{edge/.style={thick}}
     \tikzset{arc/.style={-Latex,thick}}
     \tikzstyle{every node}=[font=\footnotesize]
	 \begin{scope}[xshift=0cm,yshift=0cm,xscale=.5,yscale=.5]
	\draw[thick, fill, radius=0.06] (0,0) circle node[above left] {$p$};
	\draw[thick, fill, radius=0.06] (1,-1) circle node[left] {$x$};
	\draw[thick, fill, radius=0.06] (1,-2) circle node[below] {$c$};
	\draw[thick, fill, radius=0.06] (2,0) circle node[above right] {$y$};
	\draw[edge] (0,0) -- (1,-1);
	\draw[edge] (1,-1) -- (1,-2);
	\draw[edge] (2,0) -- (1,-1);
	\draw (1.3,-.3) node {$e$};
	\draw[thick, fill, radius=0.06] (4,0) circle node[above] {$z$};
	\draw[thick, fill, radius=0.06] (4,-2) circle node[below] {$w$};
	\draw[edge] (4,0) -- (4,-2);
	\draw[arc] (6.3,-1) -- (10.3,-1);
	\draw (8,-1) node[above,scale=.8] {$(p,x,c)\move{e}(z,w)$};
	\end{scope}
	\begin{scope}[xshift=6cm,yshift=0cm,xscale=.5,yscale=.5]
	\draw[thick, fill, radius=0.06] (0,0) circle node[above] {$p$};
	\draw[thick, fill, radius=0.06] (0,-2) circle node[below] {$c$};
	\draw[thick, fill, radius=0.06] (2,0) circle node[above left] {$y$};
	\draw[thick, fill, radius=0.06] (4,0) circle node[above right] {$z$};
	\draw[thick, fill, radius=0.06] (3,-1) circle node[left] {$x$};
	\draw[thick, fill, radius=0.06] (3,-2) circle node[below] {$w$};
	\draw[edge] (0,0) -- (0,-2);
	\draw[edge] (2,0) -- (3,-1);
	\draw[edge] (4,0) -- (3,-1);
	\draw[edge] (3,-1) -- (3,-2);
	\end{scope}
	\begin{scope}[xshift=0cm,yshift=-2cm,xscale=.5,yscale=.5]
	\draw[thick, fill, radius=0.06] (1,0) circle node[above] {$p$};
	\draw[thick, fill, radius=0.06] (1,-1) circle node[left] {$x$};
	\draw[thick, fill, radius=0.06] (0,-2) circle node[below left] {$c$};
	\draw[thick, fill, radius=0.06] (2,-2) circle node[below right] {$y$};
	\draw[edge] (1,0) -- (1,-1);
	\draw[edge] (1,-1) -- (0,-2);
	\draw[edge] (1,-1) -- (2,-2);
	\draw (1.7,-1.3) node {$e$};
	\draw[thick, fill, radius=0.06] (4,0) circle node[above] {$z$};
	\draw[thick, fill, radius=0.06] (4,-2) circle node[below] {$w$};
	\draw[edge] (4,0) -- (4,-2);
	\draw[arc] (6.3,-1) -- (10.3,-1);
	\draw (8,-1) node[above,scale=.8] {$(p,x,c)\move{e}(z,w)$};
	\end{scope}
	\begin{scope}[xshift=6cm,yshift=-2cm,xscale=.5,yscale=.5]
	\draw[thick, fill, radius=0.06] (0,0) circle node[above] {$p$};
	\draw[thick, fill, radius=0.06] (0,-2) circle node[below] {$c$};
	\draw[thick, fill, radius=0.06] (3,0) circle node[above] {$z$};
	\draw[thick, fill, radius=0.06] (3,-1) circle node[left] {$x$};
	\draw[thick, fill, radius=0.06] (2,-2) circle node[below left] {$y$};
	\draw[thick, fill, radius=0.06] (4,-2) circle node[below right] {$w$};
	\draw[edge] (0,0) -- (0,-2);
	\draw[edge] (3,0) -- (3,-1);
	\draw[edge] (3,-1) -- (2,-2);
	\draw[edge] (3,-1) -- (4,-2);
	\end{scope}
	\end{tikzpicture}
    \caption{Application of an rSPR move. \leonew{The top figure shows the case that the arc~$e$ being moved is a reticulation arc, the bottom figure the case that~$e$ is a tree arc.}}
    \label{fig:rSPRmove}
\end{figure}
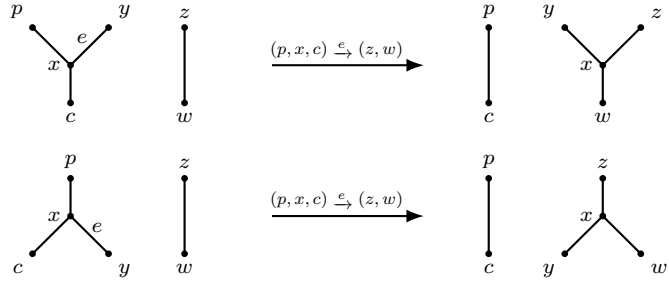


\leonew{An rSPR (or rNNI) move is \emph{valid} if the resulting graph is a network. We first show that this holds automatically \yukinew{for orchard networks} if the resulting graph has an HGT-consistent labelling. For that, we need to define HGT-consistent labellings on more general directed graphs. Note that the following definition is equivalent to Definition~\ref{def:HGT} if~$D$ is a network.}

\begin{definition}\label{def:HGTD}
\leonew{Let~$D=(V,A)$ be a directed graph that may contain parallel arcs, with indegree and outdegree at most~$2$ and total degree at most~$3$. An \emph{HGT-consistent labelling} of~$D$ is a labelling $t:V\rightarrow \RR$ such that:}
\begin{enumerate}
    \item For all arcs $(u,v)$, $t(u)\leq t(v)$ and equality is only allowed if $v$ has indegree~$2$.
    \item For each node $u$ with at least one child, there is a child $v$ of $u$ such that $t(u)<t(v)$.
    \item For each node $r$ with two parents $u$ and $v$, exactly one of $t(u)=t(r)$ and $t(v)=t(r)$ holds.
\end{enumerate}
\end{definition}

\begin{lemma}\label{lem:OrchardValid}
Let $N$ be a binary orchard network \leonew{and~$N'$ the result of an rSPR move on~$N$. If~$N'$ admits an HGT-consistent labelling, then~$N'$ is a network and hence the rSPR move is valid.}
\end{lemma}
\begin{proof}
\leonew{Suppose the rSPR move is $(p,x,c)\move{e}(z,w)$,} \yukinew{where~$e$ is the edge incident on~$x$ that is not incident on~$p$ nor~$c$.}
Let~$t$ be an HGT-consistent labelling of~$N'$.

First note that $N'$ cannot contain parallel arcs \leonew{say from~$a$ to~$b$}. Indeed, this would make $b$ into a reticulation node, and one of its parents must have the same label, so $t(a)=t(b)$. However, at least one child of $a$ must have a larger label than $a$ by Property 2 \markj{of Definition~\ref{def:HGTD}}, so $t(a)<t(b)$, a contradiction.

\leonew{In addition}, $N'$ cannot contain a directed cycle. 
Suppose for a contradiction that~$N'$ does contain a directed cycle.
Then, by Property 1 \markj{of Definition~\ref{def:HGTD}}, all nodes in this cycle must have the same label, and they must be reticulation nodes. However, reticulation nodes only have one outgoing arc, and the head of this arc must have a strictly larger label, a contradiction. \leonew{Hence, $N'$ cannot contain a directed cycle.}

\leonew{Since~$N'$ does not contain parallel arcs or directed cycles, and rSPR moves do not change the degrees, if follows that~$N'$ is a network and hence the rSPR move is valid.}
\end{proof}
If the result of a valid rNNI move on~$N$ \leonew{is a network} that is isomorphic to network~$N'$ \leonew{(respecting leaf labels)}, then we say that \emph{$N$ can be transformed into $N'$} using one rNNI move. \leonew{It is not too difficult to observe from the definition that rNNI moves are symmetric, i.e.,~$N$ can be transformed into~$N'$ using one rNNI move if and only if~$N'$ can be transformed into~$N$ using one rNNI move.}

\begin{definition}
The \emph{rNNI space of orchard networks} with $n$ leaves and $k$ reticulations is the graph $\Orch(n,k)$, whose nodes are binary orchard networks, and there is an edge between two networks if one can be transformed into the other in one rNNI move.
\end{definition}

\subsection{Connectedness of the rNNI space of orchard networks}

\leonew{In this section, we prove that $\Orch(n,k)$ is connected for all~$n$ and~$k$. The main idea of the proof is to show that we can transform any orchard network into some canonical network in which all reticulations are stacked just below the root of the network. This is formalized as follows. See also Figure~\ref{fig:NeatlyTop}.}

\begin{definition}
Let $N$ be a binary \leonew{orchard} network where~\leonew{$v_\rho$} is the child of the root. Then we say that $N$ has $k$ \emph{reticulations at the top} if it contains two directed paths $v_\rho,a_1,\ldots,a_k$ and $v_\rho,b_1,\ldots,b_k$ where all $a_i$ and $b_j$ are distinct, and a set of~$k$ reticulation arcs  $\{(x_i,y_i)\}_{i=1}^k$ where $\{x_i,y_i\}=\{a_i,b_i\}$, which are called the \emph{\leonew{horizontal} arcs at the top}. 
\leonew{In addition, there is no arc between the child of~$y_k$ and the child of~$x_k$ that is not~$y_k$ (because otherwise~$N$ would have $k+1$ reticulations at the top). 
The reticulations~$y_1,\ldots ,y_k$ are \emph{reticulations at the top}.} If $(x_i,y_i)=(a_i,b_i)$ for all $i$, then we say that $N$ has $k$ reticulations \emph{neatly} at the top.
\end{definition}


\leonew{We now show that the name ``horizontal arcs at the top'' is well chosen, i.e., that they are horizontal with respect to any HGT-consistent labelling.}

\begin{lemma}\label{lem:horizontal}
\leonew{If~$N$ is a binary orchard network with~$k$ reticulations at the top, and~$t$ an HGT-consistent labelling, then for each horizontal arc at the top $(x_i,y_i)$, holds that $t(x_i)=t(y_i)$.}
\end{lemma}
\begin{proof}
\leonew{The proof is by induction on~$i$. First consider~$i=1$. Since~$y_1$ is a reticulation, we have, by the definition of HGT-consistent labelling, that either $t(y_1)=t(x_1)$ or $t(y_1)=t(v_\rho)$. We have $t(y_1)\geq t(x_1)$ because $(x_1,y_1)$ is an arc. Furthermore, $t(x_1)>t(v_\rho)$ because~$x_1$ is not a reticulation. Hence, we have $t(y_1)\geq t(x_1)>t(v_\rho)$. So $t(y_1)=t(v_\rho)$ is not possible and we must have $t(y_1)=t(x_1)$.}

\leonew{Now assume that $t(x_i)=t(y_i)$. We will show that $t(x_{i+1})=t(y_{i+1})$. First note that we do not know whether~$x_i$ is a parent of~$x_{i+1}$ and~$y_i$ of~$y_{i+1}$ or if~$x_i$ is a parent of~$y_{i+1}$ and~$y_i$ of~$x_{i+1}$, but this does not matter for the proof since $t(x_i)=t(y_i)$. Since~$y_{i+1}$ is a reticulation, we have that either $t(y_{i+1})=t(x_{i+1})$ or $t(y_{i+1})=t(x_i)=t(y_i)$. We have $t(y_{i+1})\geq t(x_{i+1})$ because $(x_{i+1},y_{i+1})$ is an arc. Furthermore, $t(x_{i+1})>t(x_i)=t(y_i)$ because~$x_{i+1}$ is not a reticulation. Hence, we have $t(y_{i+1})\geq t(x_{i+1})>t(x_i)=t(y_i)$. So $t(y_{i+1})=t(x_i)=t(y_i)$ is not possible and we must have $t(y_{i+1})=t(x_{i+1})$.}
\end{proof}


\leonew{\emph{Reorienting} an arc~$(u,v)$ of a network~$N$ refers to modifying~$N$ into a network~$N'$ that is isomorphic to~$N$ with~$(u,v)$ replaced by~$(v,u)$.}
Note that \leonew{reorienting} any subset of the \leonew{horizontal} arcs at the top of a binary orchard network results in a binary orchard network, as the node labelling remains HGT-consistent. \leonew{The following lemma shows that, if a network has~$k$ reticulations at the top and~$k'\leq k$, then the highest~$k'$ horizontal arcs at the top (i.e. the arcs $(x_i,y_i)$ for $i=1,\ldots ,k'$) can be reoriented using one rNNI move, see Figure~\ref{fig:NeatlyTop}.}

\begin{lemma}[Lemma~5 of \cite{janssen2021heading}]\label{lem:ReticsNeatlyTop}
\leonew{Let~$N$ be a binary orchard network with~$k$ reticulations at the top and~$k'\leq k$. Then, using one rNNI move, the highest~$k'$ horizontal arcs at the top can be reoriented.}
\end{lemma}


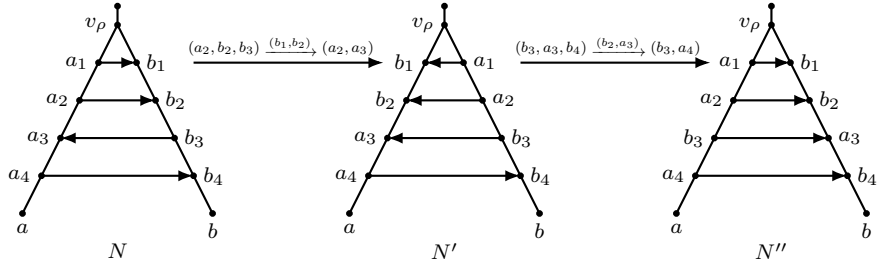
\begin{figure}
    \centering
    \begin{tikzpicture}
	 \tikzset{edge/.style={thick}}
	 \tikzset{arc/.style={-{Latex[length=2mm]},thick}}
     \tikzstyle{every node}=[font=\footnotesize]
	 \begin{scope}[xshift=0cm,yshift=0cm,xscale=.5,yscale=.5]
	\draw[thick, fill, radius=0.06] (0,.5) circle;
	\draw[thick, fill, radius=0.06] (0,0) circle node[left] {$v_\rho$};
	\draw[edge] (0,.5) -- (0,0);
	\draw[thick, fill, radius=0.06] (-2.5,-5) circle node[below] {$a$};
	\draw[thick, fill, radius=0.06] (2.5,-5) circle node[below] {$b$};
	\draw[edge] (0,0) -- (-2.5,-5);
	\draw[edge] (0,0) -- (2.5,-5);
	\draw[thick, fill, radius=0.06] (-.5,-1) circle node[left] {$a_1$};
	\draw[thick, fill, radius=0.06] (-1,-2) circle node[left] {$a_2$};
	\draw[thick, fill, radius=0.06] (-1.5,-3) circle node[left] {$a_3$};
	\draw[thick, fill, radius=0.06] (-2,-4) circle node[left] {$a_4$};
	\draw[thick, fill, radius=0.06] (.5,-1) circle node[right] {$b_1$};
	\draw[thick, fill, radius=0.06] (1,-2) circle node[right] {$b_2$};
	\draw[thick, fill, radius=0.06] (1.5,-3) circle node[right] {$b_3$};
	\draw[thick, fill, radius=0.06] (2,-4) circle node[right] {$b_4$};
	\draw[arc] (-.5,-1) -- (.5,-1);
	\draw[arc] (-1,-2) -- (1,-2);
	\draw[arc] (1.5,-3) -- (-1.5,-3);
	\draw[arc] (-2,-4) -- (2,-4);
	\draw (0,-6) node {$N$};
	\draw[arc] (2,-1) -- (7,-1);
	\draw (4.3,-1) node[above,scale=.7] {$(a_2,b_2,b_3)\move{(b_1,b_2)}(a_2,a_3)$};
	\end{scope}
	\begin{scope}[xshift=4.3cm,yshift=0cm,xscale=.5,yscale=.5]
	\draw[thick, fill, radius=0.06] (0,.5) circle;
	\draw[thick, fill, radius=0.06] (0,0) circle node[left] {$v_\rho$};
	\draw[edge] (0,.5) -- (0,0);
	\draw[thick, fill, radius=0.06] (-2.5,-5) circle node[below] {$a$};
	\draw[thick, fill, radius=0.06] (2.5,-5) circle node[below] {$b$};
	\draw[edge] (0,0) -- (-2.5,-5);
	\draw[edge] (0,0) -- (2.5,-5);
	\draw[thick, fill, radius=0.06] (-.5,-1) circle node[left] {$b_1$};
	\draw[thick, fill, radius=0.06] (-1,-2) circle node[left] {$b_2$};
	\draw[thick, fill, radius=0.06] (-1.5,-3) circle node[left] {$a_3$};
	\draw[thick, fill, radius=0.06] (-2,-4) circle node[left] {$a_4$};
	\draw[thick, fill, radius=0.06] (.5,-1) circle node[right] {$a_1$};
	\draw[thick, fill, radius=0.06] (1,-2) circle node[right] {$a_2$};
	\draw[thick, fill, radius=0.06] (1.5,-3) circle node[right] {$b_3$};
	\draw[thick, fill, radius=0.06] (2,-4) circle node[right] {$b_4$};
	\draw[arc] (.5,-1) -- (-.5,-1);
	\draw[arc] (1,-2) -- (-1,-2);
	\draw[arc] (1.5,-3) -- (-1.5,-3);
	\draw[arc] (-2,-4) -- (2,-4);
	\draw (0,-6) node {$N'$};
	\draw[arc] (2,-1) -- (7,-1);
	\draw (4.3,-1) node[above,scale=.7] {$(b_3,a_3,b_4)\move{(b_2,a_3)}(b_3,a_4)$};
	\end{scope}
	\begin{scope}[xshift=8.6cm,yshift=0cm,xscale=.5,yscale=.5]
	\draw[thick, fill, radius=0.06] (0,.5) circle;
	\draw[thick, fill, radius=0.06] (0,0) circle node[left] {$v_\rho$};
	\draw[edge] (0,.5) -- (0,0);
	\draw[thick, fill, radius=0.06] (-2.5,-5) circle node[below] {$a$};
	\draw[thick, fill, radius=0.06] (2.5,-5) circle node[below] {$b$};
	\draw[edge] (0,0) -- (-2.5,-5);
	\draw[edge] (0,0) -- (2.5,-5);
	\draw[thick, fill, radius=0.06] (-.5,-1) circle node[left] {$a_1$};
	\draw[thick, fill, radius=0.06] (-1,-2) circle node[left] {$a_2$};
	\draw[thick, fill, radius=0.06] (-1.5,-3) circle node[left] {$b_3$};
	\draw[thick, fill, radius=0.06] (-2,-4) circle node[left] {$a_4$};
	\draw[thick, fill, radius=0.06] (.5,-1) circle node[right] {$b_1$};
	\draw[thick, fill, radius=0.06] (1,-2) circle node[right] {$b_2$};
	\draw[thick, fill, radius=0.06] (1.5,-3) circle node[right] {$a_3$};
	\draw[thick, fill, radius=0.06] (2,-4) circle node[right] {$b_4$};
	\draw[arc] (-.5,-1) -- (.5,-1);
	\draw[arc] (-1,-2) -- (1,-2);
	\draw[arc] (-1.5,-3) -- (1.5,-3);
	\draw[arc] (-2,-4) -- (2,-4);
	\draw (0,-6) node {$N''$};
	\end{scope}
	\end{tikzpicture}
    \caption{Left: A network~$N$ on a set of taxa~$\{a,b\}$ with four reticulations at the top. The network~$N$ contains a triangle at~$v_\rho$ with arcs~$(v_\rho,a_1), (v_\rho,b_1),$ and~$(a_1,b_1)$, \leonew{where~$v_\rho$ is the child of the root}.
    Middle: A binary network~$N'$ obtained by applying the valid rNNI move~$(a_2,b_2,b_3)\move{(b_1,b_2)}(a_2,a_3)$ on~$N$. 
    This illustrates \Cref{lem:ReticsNeatlyTop}, where the highest \leonew{two horizontal arcs} at the top are \leonew{reoriented}.
    Right: A network~$N''$ by applying the valid rNNI move~$(b_3,a_3,b_4)\move{(b_2,a_3)}(b_3,a_4)$ to~$N'$.
    The highest \leonew{three horizontal arcs} at the top are \leonew{reoriented}.
    The network~$N''$ has four reticulations neatly at the top.}
    \label{fig:NeatlyTop}
\end{figure}

\leonew{We will prove that we can transform any orchard network into a network in which all reticulations are at the top. We first consider the case of moving a reticulation~$r$ that is part of a \emph{triangle}, i.e. when there are arcs $(v,p)$, $(v,r)$, and $(p,r)$. In this case, we also say that the triangle is \emph{at}~$v$. Observe that, by the same argument as in the first part of the proof of Lemma~\ref{lem:horizontal}, $t(p)=t(r)$ for any HGT-consistent labelling~$t$. The following lemma shows that such a reticulation~$r$ can either be moved up or directly to the top.}

\begin{lemma}\label{lem:TriangleToTop}
\leonew{Let~$N$ be a binary orchard network with $k<r(N)$ reticulations at the top. Let~$t$ be an HGT-consistent labelling and~$r$ a reticulation not at the top minimizing~$t(r)$. Suppose that~$r$ is part of a triangle $(w,p)$, $(w,r)$, $(p,r)$. Then we can either reduce the number of nodes above~$r$ by~$1$, in~$2$ rNNI moves, or transform~$N$ into an orchard network with~$k+1$ reticulations at the top, in at most~$4$ rNNI moves.}
\end{lemma}
\begin{proof}
Let~$t$ be an HGT-consistent labelling for~$N$.
Let~$q$ be the parent of~$w$.
First suppose that~$q$ is not an endpoint of a horizontal arc at the top. Then~$q$ is a tree node by our choice of~$r$. Let~$v$ be its child other than~$w$.
In that case, the graph~$N'$ obtained by rNNI move $(q,w,p)\move{(w,r)}(q,v)$ admits an HGT-consistent labelling (see \Cref{fig:7TriangleUp})
\[t'(x)=\begin{cases}
\min\{t(w),t(v)\}-\epsilon &\mbox{ if } x=w;\\
t(x) &\mbox{ otherwise,} 
\end{cases}\]
where $\epsilon>0$ is small enough.
In particular, choose~$\epsilon$ so that~$t'(q) < t'(w)$. Then,~$N'$ is an orchard network by Lemma~\ref{lem:OrchardValid} and \Cref{thm:OrchIFFHGT}.
Let~$c$ be the child of~$r$ in~$N'$, and we apply to~$N'$ the move $(q,w,r)\move{(w,v)}(r,c)$ to obtain the graph~$N''$.
The graph~$N''$ is an orchard network, since it admits the HGT-consistent labelling
\[t''(x)=\begin{cases}
t'(w)-\epsilon &\mbox{ if } x\in\{p,r\};\\
t'(x) &\mbox{ otherwise,} 
\end{cases}\]
where again,~$\epsilon>0$ is very small.
Observe that~$N''$ 
contains a triangle at~$q$ consisting of the arcs $(q,p)$, $(q,r)$, and $(p,r)$. Hence, we have reduced the number of nodes above~$r$ by~$1$, using~$2$ rNNI moves, and~$r$ \markj{is still a reticulation not at the top minimizing $t''(r)$.}
\medskip

\begin{figure}
    \centering
    \begin{tikzpicture}
	 \tikzset{edge/.style={thick}}
     \tikzset{arc/.style={-{Latex[length=2mm]},thick}}
     \tikzstyle{every node}=[font=\footnotesize]
	 \begin{scope}[xshift=0cm,yshift=0cm,xscale=.5,yscale=.5]
	\draw[edge] (0,.5) -- (0,0);
    \draw[thick, fill, radius=0.06] (0,0) circle node[left] {$q$};
	\draw[thick, fill, radius=0.06] (1,-1) circle node[below] {$v$};
	\draw[edge] (0,0) -- (1,-1);
	\draw[thick, fill, radius=0.06] (-1.5,-1.5) circle node[left] {$w$};
	\draw[edge] (0,0) -- (-1.5,-1.5);
	\draw[thick, fill, radius=0.06] (-2.5,-2.5) circle node[left] {$p$};
	\draw[edge] (-1.5,-1.5) -- (-2.5,-2.5);
	\draw[edge] (-2.5,-2.5) -- (-3,-3);
	\draw[thick, fill, radius=0.06] (-0.5,-2.5) circle node[right] {$r$};
	\draw[edge] (-1.5,-1.5) -- (-0.5,-2.5);
	\draw[edge] (-0.5,-2.5) -- (-0.5,-3.2);
	\draw[thick, fill, radius=0.06] (-0.5,-3.2) circle node[below] {$c$};
	\draw[arc] (-2.5,-2.5) -- (-0.5,-2.5);
	\draw (0,-4.5) node {$N$};
	\draw[arc] (1.8,-2) -- (6.2,-2);
	\draw (3.8,-2) node[above,scale=.8] {$(q,w,p)\move{(w,r)}(q,v)$};
	\end{scope}
	\begin{scope}[xshift=4.8cm,yshift=0cm,xscale=.5,yscale=.5]
	\draw[edge] (0,.5) -- (0,0);
    \draw[thick, fill, radius=0.06] (0,0) circle node[left] {$q$};
	\draw[thick, fill, radius=0.06] (1,-1) circle node[right] {$v$};
	\draw[edge] (0,0) -- (1,-1);
	\draw[thick, fill, radius=0.06] (.7,-.7) circle;
	\draw (.6,-.8) node[above right] {$w$};
	\draw[edge] (0,0) -- (-1.5,-1.5);
	\draw[thick, fill, radius=0.06] (-2.5,-2.5) circle node[left] {$p$};
	\draw[edge] (-1.5,-1.5) -- (-2.5,-2.5);
	\draw[edge] (-2.5,-2.5) -- (-3,-3);
	\draw[thick, fill, radius=0.06] (-0.5,-2.5) circle node[right] {$r$};
	\draw[edge] (.7,-.7) -- (-0.5,-2.5);
	\draw[edge] (-0.5,-2.5) -- (-0.5,-3.2);
	\draw[arc] (-2.5,-2.5) -- (-0.5,-2.5);
	\draw[arc] (1.8,-2) -- (6.2,-2);
	\draw[thick, fill, radius=0.06] (-0.5,-3.2) circle node[below] {$c$};
	\draw (0,-4.5) node {$N'$};
	\draw (3.8,-2) node[above,scale=.8] {$(q,w,r)\move{(w,v)}(r,c)$};
	\end{scope}
	\begin{scope}[xshift=9.6cm,yshift=0cm,xscale=.5,yscale=.5]
	\draw[edge] (0,.5) -- (0,0);
    \draw[thick, fill, radius=0.06] (0,0) circle node[left] {$q$};
	\draw[thick, fill, radius=0.06] (1,-1) circle node[right] {$v$};
	\draw[edge] (0,0) -- (1,-1);
	\draw[thick, fill, radius=0.06] (.7,-.7) circle;
	\draw (.6,-.8) node[above right] {$w$};
	\draw[edge] (0,0) -- (-1.5,-1.5);
	\draw[thick, fill, radius=0.06] (-.4,-.4) circle node[left] {$p$};
	\draw[edge] (-1.5,-1.5) -- (-2.5,-2.5);
	\draw[edge] (-2.5,-2.5) -- (-3,-3);
	\draw[thick, fill, radius=0.06] (.4,-.4) circle;
	\draw (.3,-.5) node[above right] {$r$};
	\draw[edge] (.7,-.7) -- (-0.5,-3.2);
	\draw[arc] (-.4,-.4) -- (.4,-.4);
	\draw (0,-4.5) node {$N''$};
	\draw[thick, fill, radius=0.06] (-0.5,-3.2) circle node[below] {$c$};
	\end{scope}
	\end{tikzpicture}
    \caption{The first case in the proof of \Cref{lem:TriangleToTop}, where~$q$ is not an endpoint of a horizontal arc at the top.
    Two rNNI moves are applied to the orchard network, which results in an orchard network where the triangle is at~$q$. \leonew{The heights of the nodes represent an HGT-consistent labelling.}}
    \label{fig:7TriangleUp}
\end{figure}
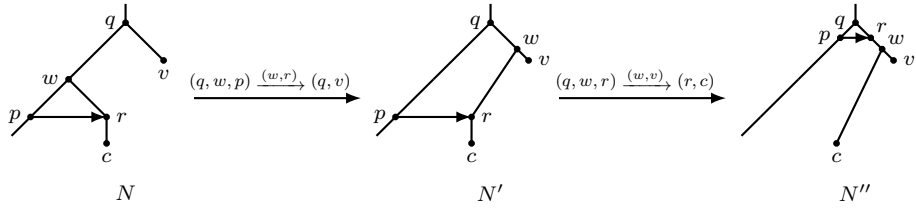

Now suppose that~$q$ is an endpoint of a horizontal arc at the top. Then we can use the following rNNI moves to move $(p,r)$ to the top. First, if~$q$ is not a reticulation, reorient the reticulation arc leaving~$q$ so that~$q$ becomes a reticulation, using one rNNI move (Lemma~\ref{lem:ReticsNeatlyTop}). Let~$s$ be the parent of~$q$ with $t(q)=t(s)$, let~$v$ be the child of~$s$ that is not~$q$ and let~$c$ be the child of~$r$ (see Figure~\ref{fig:ToTheTop}).
Apply the following three moves: $(q,w,p)\move{(w,r)}(s,q)$, $(s,w,q)\move{(w,r)}(s,v)$, and $(s,w,r)\move{(w,v)}(r,c)$. 
To see that the graphs after the three moves are orchard networks, note that the following are HGT-consistent labellings of the respective resulting graphs:

\[t'(x)=\begin{cases}
\min\{t(w),t(v)\} - \epsilon &\mbox{ if } x\in\{q,w\};\\
t(x) &\mbox{ otherwise,} 
\end{cases}\]
\[t''(x)=\begin{cases}
\leonew{t'(s)} &\mbox{ if } x=q;\\
\leonew{t'(s)} + \delta &\mbox{ if } x=w;\\
t'(x) &\mbox{ otherwise,} 
\end{cases}\]
and
\[t'''(x)=\begin{cases}
\leonew{t''(w)-}\gamma &\mbox{ if } x\in\{p,r\};\\
t''(x) &\mbox{ otherwise,} 
\end{cases}\]

where~$0 <\gamma <\delta <\epsilon$ are very small.
Observe that reticulation~$r$ is now at the top and~$(p,r)$ a horizontal arc at the top (see \Cref{fig:ToTheTop}). Hence, the resulting graph \yukinew{obtained using at most~$4$ rNNI moves,} is an orchard network with~$k+1$ reticulations at the top.

\begin{figure}
    \centering
    \begin{tikzpicture}
	 \tikzset{edge/.style={thick}}
     \tikzset{arc/.style={-{Latex[length=2mm]},thick}}
     \tikzstyle{every node}=[font=\footnotesize]
	 \begin{scope}[xshift=0cm,yshift=0cm,xscale=.5,yscale=.5]
	\draw[edge] (-1.5,1) -- (-1.5,.7);
    \draw[thick, fill, radius=0.06] (-1.5,.7) circle node[left] {$q$};
	\draw[thick, fill, radius=0.06] (1,-1) circle node[below] {$v$};
	\draw[thick, fill, radius=0.06] (1,.7) circle node[right] {$s$};
	\draw[edge] (1,1) -- (1,.7);
	\draw[edge] (1,.7) -- (1,-1);
	\draw[arc] (1,.7) -- (-1.5,.7);
	\draw[thick, fill, radius=0.06] (-1.5,-1.5) circle node[left] {$w$};
	\draw[edge] (-1.5,.7) -- (-1.5,-1.5);
	\draw[thick, fill, radius=0.06] (-2.5,-2.5) circle node[left] {$p$};
	\draw[edge] (-1.5,-1.5) -- (-2.5,-2.5);
	\draw[edge] (-2.5,-2.5) -- (-3,-3);
	\draw[thick, fill, radius=0.06] (-0.5,-2.5) circle node[right] {$r$};
	\draw[edge] (-1.5,-1.5) -- (-0.5,-2.5);
	\draw[edge] (-0.5,-2.5) -- (-0.5,-3.2);
	\draw[thick, fill, radius=0.06] (-0.5,-3.2) circle node[below] {$c$};
	\draw[arc] (-2.5,-2.5) -- (-0.5,-2.5);
	\draw[arc] (-.7,-4.8) -- (4,-4.8);
	\draw (1.5,-4.8) node[above,scale=.8] {$(q,w,p)\move{(w,r)}(s,q)$};
	\end{scope}
	\begin{scope}[xshift=3cm,yshift=0cm,xscale=.5,yscale=.5]
	\draw[edge] (-1.5,1) -- (-1.5,-.3);
    \draw[thick, fill, radius=0.06] (-1.5,-.3) circle node[left] {$q$};
	\draw[thick, fill, radius=0.06] (1,-1) circle node[below] {$v$};
	\draw[thick, fill, radius=0.06] (1,.7) circle node[right] {$s$};
	\draw[edge] (1,1) -- (1,.7);
	\draw[edge] (1,.7) -- (1,-1);
	\draw[edge] (1,.7) -- (0,-.3);
	\draw[thick, fill, radius=0.06] (0,-.3) circle node[right] {$w$};
	\draw[arc] (0,-.3) -- (-1.5,-.3);
	\draw[edge] (-1.5,-.3) -- (-2.5,-2.5);
	\draw[thick, fill, radius=0.06] (-2.5,-2.5) circle node[left] {$p$};
	\draw[edge] (-2.5,-2.5) -- (-3,-3);
	\draw[thick, fill, radius=0.06] (-0.5,-2.5) circle node[right] {$r$};
	\draw[edge] (0,-.3) -- (-0.5,-2.5);
	\draw[edge] (-0.5,-2.5) -- (-0.5,-3.2);
	\draw[thick, fill, radius=0.06] (-0.5,-3.2) circle node[below] {$c$};
	\draw[arc] (-2.5,-2.5) -- (-0.5,-2.5);
	\draw[arc] (-.7,-4.8) -- (4,-4.8);
	\draw (1.5,-4.8) node[above,scale=.8] {$(s,w,q)\move{(w,r)}(s,v)$};
	\end{scope}
	\begin{scope}[xshift=6cm,yshift=0cm,xscale=.5,yscale=.5]
	\draw[edge] (-1.5,1) -- (-1.5,.7);
    \draw[thick, fill, radius=0.06] (-1.5,.7) circle node[left] {$q$};
	\draw[thick, fill, radius=0.06] (1,-1) circle node[below] {$v$};
	\draw[thick, fill, radius=0.06] (1,.7) circle node[right] {$s$};
	\draw[edge] (1,1) -- (1,.7);
	\draw[edge] (1,.7) -- (1,-1);
	\draw[thick, fill, radius=0.06] (1,-.3) circle node[right] {$w$};
	\draw[arc] (1,.7) -- (-1.5,.7);
	\draw[edge] (-1.5,.7) -- (-2.5,-2.5);
	\draw[thick, fill, radius=0.06] (-2.5,-2.5) circle node[left] {$p$};
	\draw[edge] (-2.5,-2.5) -- (-3,-3);
	\draw[thick, fill, radius=0.06] (-0.5,-2.5) circle node[right] {$r$};
	\draw[edge] (1,-.3) -- (-0.5,-2.5);
	\draw[edge] (-0.5,-2.5) -- (-0.5,-3.2);
	\draw[thick, fill, radius=0.06] (-0.5,-3.2) circle node[below] {$c$};
	\draw[arc] (-2.5,-2.5) -- (-0.5,-2.5);
	\draw[arc] (-.7,-4.8) -- (4,-4.8);
	\draw (1.5,-4.8) node[above,scale=.8] {$(s,w,r)\move{(w,v)}(r,c)$};
	\end{scope}
	\begin{scope}[xshift=9cm,yshift=0cm,xscale=.5,yscale=.5]
	\draw[edge] (-1.5,1) -- (-1.5,.7);
    \draw[thick, fill, radius=0.06] (-1.5,.7) circle node[left] {$q$};
	\draw[thick, fill, radius=0.06] (1,-1) circle node[below] {$v$};
	\draw[thick, fill, radius=0.06] (1,.7) circle node[right] {$s$};
	\draw[edge] (1,1) -- (1,.7);
	\draw[edge] (1,.7) -- (1,-1);
	\draw[thick, fill, radius=0.06] (1,-.3) circle node[right] {$w$};
	\draw[arc] (1,.7) -- (-1.5,.7);
	\draw[edge] (-1.5,.7) -- (-1.5,.1);
	\draw[thick, fill, radius=0.06] (-1.5,.1) circle node[left] {$p$};
	\draw[edge] (-1.5,.1) -- (-3,-3);
	\draw[thick, fill, radius=0.06] (1,.1) circle node[right] {$r$};
	\draw[edge] (1,-.3) -- (-0.5,-3.2);
	\draw[thick, fill, radius=0.06] (-0.5,-3.2) circle node[below] {$c$};
	\draw[arc] (-1.5,.1) -- (1,.1);
	\end{scope}
	\end{tikzpicture}
    \caption{The second case in the proof of \Cref{lem:TriangleToTop}, where~$q$ is a reticulation at the top. Three rNNI moves are applied to the orchard network, which results in an orchard network where reticulation~$r$ is also at the top. If~$q$ is not a reticulation in the original network, one extra rNNI move is applied first to make~$q$ a reticulation.}
    \label{fig:ToTheTop}
\end{figure}
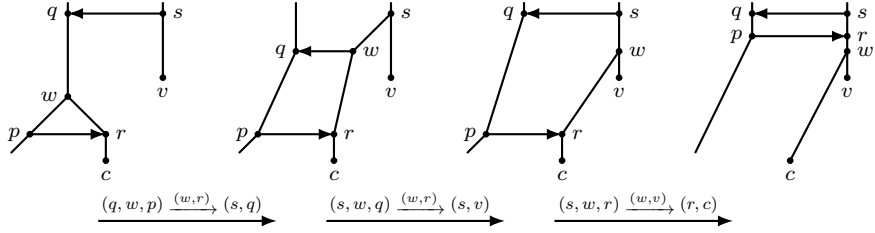

\end{proof}

\leonew{We are now ready to prove that we can move all reticulations to the top.}

\begin{lemma}\label{lem:ReticToTop}
Let $N$ be a binary orchard network with $k<r(N)$ reticulations at the top and~$n$ leaves. 
Then, using at most $2n$ rNNI moves, $N$ can be transformed into a \leonew{binary orchard} network with $k+1$ reticulations at the top.
\end{lemma}
\begin{proof}
Let $N$ be an orchard network with HGT-consistent labelling $t$ such that
two nodes have the same label \leonew{only} if they are a parent-child pair for which the parent is a tree node and the child a reticulation (which exists by \Cref{lem:OrchardThenHGTConsistent}).
\leonew{Let~$r$ be a reticulation not at the top minimizing~$t(r)$.}
Let $p$ be the parent of $r$ with $t(r)=t(p)$ and $u$ the other parent of $r$. \leonew{Observe that, since $t(r)=t(p)$, $p$ must have a second child with larger label. Hence,~$p$ cannot be a reticulation. Let~$q$ be the parent of~$p$.} In addition, let $y$ be \leonew{a} lowest \leonew{common} ancestor (LCA) of $u$ and $q$ in $N$.


First suppose $t(u)>t(q)$ (and hence~$u\neq q$). Note that $u$ is either a tree node or a reticulation at the top, by the choice of~$r$. If $u$ is a tree node, we can apply the move \leonew{$(u,r,c)\move{(p,r)}(s,u)$} 
where~$s$ is the parent of~$u$, $v$ is the child of $u$ other than~$r$, and $c$ is the child of $r$ (Figure~\ref{fig:MoveReticUp}). Then
\[t'(x)=
\begin{cases}
t(u) &\mbox{ if } x\in\{p,r\},\\
t(u)+\epsilon &\mbox{ if } x=u,\\
t(x) &\mbox{ otherwise,} 
\end{cases}
\] 
with $\epsilon>0$ small enough, is an HGT labelling of the resulting network, which is therefore an orchard network by \Cref{lem:OrchardValid}. 
In the resulting network, the number of nodes above~$r$ is reduced by~$1$.

If $u$ is a reticulation at the top, then $y$ is the parent of $u$ with $t(y)=t(u)$. However, because~$y$ is above~$q$, we have~$t(y)\leq t(q)$ and so $t(u)\leq t(q)$,
contradicting our assumption $t(u)>t(q)$.
\medskip

The case that $t(u)<t(q)$ is symmetric. We can argue as in the previous paragraph but replacing $u,r$ by $q,p$ respectively, \leonew{see Figure~\ref{fig:MoveReticUp-b}.}


\medskip

The last case is that $t(u)=t(q)$. We claim that~$u=q=y$.
Otherwise, by our choice of the labelling~$t$, we either have an arc~$(q,u)$ and~$u$ is a reticulation, or we have an arc~$(u,q)$ and~$q$ is a reticulation.
In the first scenario,~$u$ must be a reticulation at the top by our choice of~$r$. However, since~$r$ is a child of~$u$ and~$p$ a child of~$q$, it follows that~$r$ is also a reticulation at the top, a contradiction. In the second scenario,~$q$ must be a reticulation at the top and we again obtain a contradiction, by the same reasoning.
Therefore, we conclude that~$u=q=y$. This means that there is a triangle $(y,r)$, $(y,p)$, $(p,r)$ and we can apply Lemma~\ref{lem:TriangleToTop} to either increase the number of reticulations at the top by~$1$ in at most~$4$ rNNI moves, or reduce the number of nodes above~$r$ by~$1$ in~$2$ rNNI moves.
\medskip

\leonew{It remains to bound the number of moves. We say that a node is a \emph{base node} if it is \markj{an internal} node of the base tree obtained by deleting all arcs $(u,v)$ with~$t(u)=t(v)$ and suppressing indegree-1 outdegree-1 nodes. Hence, \markj{an internal} node is a base node precisely if it has no parent or child with the same label in the HGT-consistent labelling~$t$. Define~$a_b(v)$ as the number of base nodes above~$v$. Hence, a reticulation~$r$ is at the top if and only if \markj{$a_b(r)=2$}. Since a rooted phylogenetic tree with~$n$ leaves has~\markj{$n$} internal nodes, the network~$N$ has~\markj{$n$} base nodes.}

\leonew{In the proof above, we consider a reticulation~$r$ not at the top minimizing~$t(r)$. Consequently, all nodes above~$r$ are base nodes or endpoints of horizontal arcs at the top. If~$r$ is not part of a triangle, we reduce the number of nodes above~$r$, and hence~$a_b(r)$, by~$1$ in one rNNI move. If~$r$ is part of a triangle, we achieve the same in~$2$ rNNI moves, or move~$r$ to the top in at most~$4$ rNNI moves. We end with~$a_b(r)=\markj{2}$. Hence, we need at most~$2$ rNNI moves per base node, except the \markj{root and its child}, plus at most~$4$ additional moves. The total number of rNNI moves required to move~$r$ to the top is at most $2(n-2) + 4 = 2n$.}


\begin{figure}
    \centering
    \begin{tikzpicture}
	 \tikzset{edge/.style={thick}}
     \tikzset{arc/.style={-{Latex[length=2mm]},thick}}
     \tikzstyle{every node}=[font=\footnotesize]
	 \begin{scope}[xshift=0cm,yshift=0cm,xscale=.5,yscale=.5]
	\draw[thick, fill, radius=0.06] (0,0) circle node[above] {$q$};
    \draw[thick, fill, radius=0.06] (2,0) circle node[above] {$s$};
    \draw[thick, fill, radius=0.06] (0,-2) circle node[left] {$p$};
    \draw[edge] (0,0) -- (0,-2);
    \draw[thick, fill, radius=0.06] (2,-1) circle node[right] {$u$};
    \draw[thick, fill, radius=0.06] (3,-3) circle node[below] {$v$};
    \draw[edge] (2,-1) -- (3,-3);
    \draw[edge] (2,0) -- (2,-1);
    \draw[thick, fill, radius=0.06] (1,-2) circle node[right] {$r$};
    \draw[edge] (2,-1) -- (1,-2);
    \draw[arc] (0,-2) -- (1,-2);
    \draw[thick, fill, radius=0.06] (1,-3) circle node[below] {$c$};
    \draw[edge] (1,-2) -- (1,-3);
    \draw[edge] (0,-2) -- (0,-3);
	\draw[arc] (3.4,-2) -- (8,-2);
	\draw (5.5,-2) node[above,scale=.8] {$(u,r,c)\move{(p,r)}(s,u)$};
	\draw (5.5,-4.5) node {(a)};
	\end{scope}
	\begin{scope}[xshift=4.5cm,yshift=0cm,xscale=.5,yscale=.5]
	\draw[thick, fill, radius=0.06] (0,0) circle node[above] {$q$};
    \draw[thick, fill, radius=0.06] (2,0) circle node[above] {$s$};
    \draw[thick, fill, radius=0.06] (0,-1) circle node[left] {$p$};
    \draw[edge] (0,0) -- (0,-2);
    \draw[thick, fill, radius=0.06] (2,-1.5) circle node[right] {$u$};
    \draw[thick, fill, radius=0.06] (3,-3) circle node[below] {$v$};
    \draw[edge] (2,-1.5) -- (3,-3);
    \draw[edge] (2,0) -- (2,-1.5);
    \draw[thick, fill, radius=0.06] (2,-1) circle node[right] {$r$};
    \draw[edge] (2,-1.5) -- (1,-3);
    \draw[arc] (0,-1) -- (2,-1);
    \draw[thick, fill, radius=0.06] (1,-3) circle node[below] {$c$};
    \draw[edge] (0,-2) -- (0,-3);
	\end{scope}
	\begin{scope}[xshift=8cm,yshift=0cm,xscale=.5,yscale=.5]
	\draw[thick, fill, radius=0.06] (0,-1) circle node[left] {$q$};
    \draw[thick, fill, radius=0.06] (2,0) circle node[right] {$u$};
    \draw[edge] (2,1) -- (2,0);
    \draw[thick, fill, radius=0.06] (1,0) circle node[left] {$y$};
    \draw[edge,dotted] (1,0) -- (0,-1);
    \draw[edge] (1,1) -- (1,0);
    \draw[arc] (1,0) -- (2,0);
    \draw[thick, fill, radius=0.06] (0,-2) circle node[left] {$p$};
    \draw[edge] (0,-1) -- (0,-2);
    \draw[thick, fill, radius=0.06] (1,-2) circle node[right] {$r$};
    \draw[edge] (2,0) -- (1,-2);
    \draw[arc] (0,-2) -- (1,-2);
    \draw[thick, fill, radius=0.06] (1,-3) circle node[below] {$c$};
    \draw[edge] (1,-2) -- (1,-3);
    \draw[edge] (0,-2) -- (0,-3);
	\draw (1,-4.5) node {(b)};
	\end{scope}
	\end{tikzpicture}
    \caption{Illustration of the proof of \Cref{lem:ReticToTop} for the case~$t(u)>t(q)$.
     (a) When~$u$ is a tree node, the head of the reticulation arc~$(p,r)$ is moved up.
     (b) When~$u$ is a reticulation at the top, an LCA~$y$ of~$u$ and~$q$ is a parent of~$u$ with~$t(y)=t(u)$, leading to a contradiction to the assumption $t(u)>t(q)$.
    }
    \label{fig:MoveReticUp}
\end{figure}
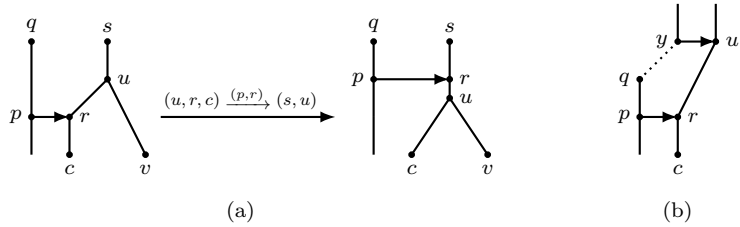
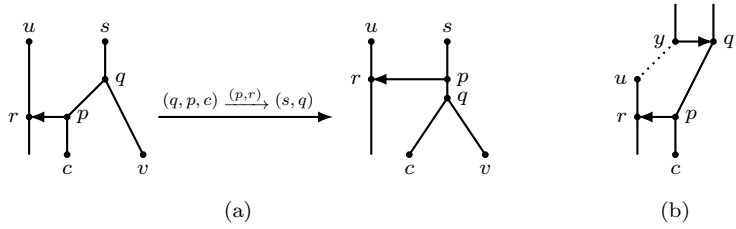
\begin{figure}
    \centering
    \begin{tikzpicture}
	 \tikzset{edge/.style={thick}}
     \tikzset{arc/.style={-{Latex[length=2mm]},thick}}
     \tikzstyle{every node}=[font=\footnotesize]
	 \begin{scope}[xshift=0cm,yshift=0cm,xscale=.5,yscale=.5]
	\draw[thick, fill, radius=0.06] (0,0) circle node[above] {$u$};
    \draw[thick, fill, radius=0.06] (2,0) circle node[above] {$s$};
    \draw[thick, fill, radius=0.06] (0,-2) circle node[left] {$r$};
    \draw[edge] (0,0) -- (0,-2);
    \draw[thick, fill, radius=0.06] (2,-1) circle node[right] {$q$};
    \draw[thick, fill, radius=0.06] (3,-3) circle node[below] {$v$};
    \draw[edge] (2,-1) -- (3,-3);
    \draw[edge] (2,0) -- (2,-1);
    \draw[thick, fill, radius=0.06] (1,-2) circle node[right] {$p$};
    \draw[edge] (2,-1) -- (1,-2);
    \draw[arc] (1,-2) -- (0,-2);
    \draw[thick, fill, radius=0.06] (1,-3) circle node[below] {$c$};
    \draw[edge] (1,-2) -- (1,-3);
    \draw[edge] (0,-2) -- (0,-3);
	\draw[arc] (3.4,-2) -- (8,-2);
	\draw (5.5,-2) node[above,scale=.8] {$(q,p,c)\move{(p,r)}(s,q)$};
	\draw (5.5,-4.5) node {(a)};
	\end{scope}
	\begin{scope}[xshift=4.5cm,yshift=0cm,xscale=.5,yscale=.5]
	\draw[thick, fill, radius=0.06] (0,0) circle node[above] {$u$};
    \draw[thick, fill, radius=0.06] (2,0) circle node[above] {$s$};
    \draw[thick, fill, radius=0.06] (0,-1) circle node[left] {$r$};
    \draw[edge] (0,0) -- (0,-2);
    \draw[thick, fill, radius=0.06] (2,-1.5) circle node[right] {$q$};
    \draw[thick, fill, radius=0.06] (3,-3) circle node[below] {$v$};
    \draw[edge] (2,-1.5) -- (3,-3);
    \draw[edge] (2,0) -- (2,-1.5);
    \draw[thick, fill, radius=0.06] (2,-1) circle node[right] {$p$};
    \draw[edge] (2,-1.5) -- (1,-3);
    \draw[arc] (2,-1) -- (0,-1);
    \draw[thick, fill, radius=0.06] (1,-3) circle node[below] {$c$};
    \draw[edge] (0,-2) -- (0,-3);
	\end{scope}
	\begin{scope}[xshift=8cm,yshift=0cm,xscale=.5,yscale=.5]
	\draw[thick, fill, radius=0.06] (0,-1) circle node[left] {$u$};
    \draw[thick, fill, radius=0.06] (2,0) circle node[right] {$q$};
    \draw[edge] (2,1) -- (2,0);
    \draw[thick, fill, radius=0.06] (1,0) circle node[left] {$y$};
    \draw[edge,dotted] (1,0) -- (0,-1);
    \draw[edge] (1,1) -- (1,0);
    \draw[arc] (1,0) -- (2,0);
    \draw[thick, fill, radius=0.06] (0,-2) circle node[left] {$r$};
    \draw[edge] (0,-1) -- (0,-2);
    \draw[thick, fill, radius=0.06] (1,-2) circle node[right] {$p$};
    \draw[edge] (2,0) -- (1,-2);
    \draw[arc] (1,-2) -- (0,-2);
    \draw[thick, fill, radius=0.06] (1,-3) circle node[below] {$c$};
    \draw[edge] (1,-2) -- (1,-3);
    \draw[edge] (0,-2) -- (0,-3);
	\draw (1,-4.5) node {(b)};
	\end{scope}
	\end{tikzpicture}
    \caption{\leonew{Illustration of the proof of \Cref{lem:ReticToTop} for the case~$t(q)>t(u)$.
     (a) When~$q$ is a tree node, the tail of the reticulation arc~$(p,r)$ is moved up.
     (b) When~$q$ is a reticulation at the top, an LCA~$y$ of~$u$ and~$q$ is a parent of~$q$ with~$t(y)=t(q)$, leading to a contradiction to the assumption $t(q)>t(u)$.}
    }
    \label{fig:MoveReticUp-b}
\end{figure}
\end{proof}

\leonew{The next step is to move all but one of the leaves to one side of the network, see Figure~\ref{fig:canonicalNetworkConnectedness}.}

\begin{lemma}\label{lem:PendantTrees}
Let $N$ be a binary network on~$n$ leaves with~$r(N)$ reticulations at the top. Let~$\leonew{l}$ be a leaf below the \leonew{head} of the lowest horizontal arc at the top. Then, using at most $2n-4$ rNNI moves, $N$ can be transformed into a network with $r(N)$ reticulations at the top, where~$l$ is the only leaf below the \leonew{head} of the lowest horizontal arc at the top.
\end{lemma}
\begin{proof}
Let $(x,y)$ be the lowest horizontal arc at the top.
Let~\leonew{$u_0$} be the child of~$x$ other than~$y$.
Let $y,u_1,\ldots,u_m=l$ be the unique directed path from $y$ to $l$. Let~$v_i$ be the child of~$u_i$ other than~$u_{i+1}$ or~$l$, for $i=1,\ldots ,\leonew{m-1}$. We apply the following sequence of rNNI moves: $(y,u_i,u_{i+1})\move{(u_i,v_i)}(x,y)$ and $(x,u_i,y)\move{(u_i,v_i)}(x,\leonew{u_{i-1}})$ for all $i=1,\ldots,\leonew{m-1}$. \leonew{See Figure~\ref{fig:movetootherside}.} 
\leonew{
It can easily be checked that the following maps~$t_1^{(i)}$ and~$t_2^{(i)}$ are HGT-consistent labellings of the graphs obtained after the first and, respectively, second rNNI move, for~$i=1,\ldots,m-1$. Letting~$t=t_2^{(0)}$ denote an HGT-labelling for~$N$, we have
\[t_1^{(i)}(z)=\begin{cases}
t(x) + \epsilon &\mbox{ if } z\in\{u_i,y\};\\
t_2^{(i-1)}(z) &\mbox{ otherwise,} 
\end{cases}\]
and
\[t_2^{(i)}(z)=\begin{cases}
t(x) &\mbox{ if } z=y;\\
\min\{t(u_0),t(u_1)\} - i\epsilon &\mbox{ if } z=u_i;\\
t_1^{(i)}(z) &\mbox{ otherwise,} 
\end{cases}\]
where~$\epsilon>0$ is very small. Hence, all intermediate graphs, as well as the final one, are orchard networks. The final network \leonew{(see Figure~\ref{fig:canonicalNetworkConnectedness})} still has $r(N)$ reticulations at the top, and~$l$ is the child of~$y$.
}
Moreover, as $m\leq n-2$, the sequence of moves contains at most $2(n-2)=2n-4$ rNNI moves.
\end{proof}

\begin{figure}
    \centering
    \begin{tikzpicture}
	 \tikzset{edge/.style={thick}}
     \tikzset{arc/.style={-{Latex[length=2mm]},thick}}
     \tikzstyle{every node}=[font=\footnotesize]
	\begin{scope}[xshift=0cm,yshift=0cm,xscale=.5,yscale=.5]
	\draw[thick, fill, radius=0.06] (0,.5) circle;
	\draw[thick, fill, radius=0.06] (0,0) circle;
	\draw[edge] (0,.5) -- (0,0);
	\draw[thick, fill, radius=0.06] (-2.5,-2.5) circle node[below] {$l=u_m$};
	\draw[thick, fill, radius=0.06] (2.5,-2.5) circle;
	\draw[edge] (0,0) -- (-1,-1);
	\draw[edge,dashed] (-1,-1) -- (-2,-2);
	\draw[edge] (-2,-2) -- (-2.5,-2.5);
	\draw[edge] (0,0) -- (1,-1);
	\draw[edge,dashed] (1,-1) -- (2,-2);
	\draw[edge] (2,-2) -- (2.5,-2.5);
    \draw[edge,dashed] (2.5,-2.5) -- (4,-4);	
	\draw[edge] (4,-4) -- (5.5,-5.5);
	
	\draw[thick, fill, radius=0.06] (-.5,-.5) circle;
	\draw[thick, fill, radius=0.06] (-1,-1) circle;
	\draw[thick, fill, radius=0.06] (-2,-2) circle node[left] {$y$};
	\draw[thick, fill, radius=0.06] (.5,-.5) circle;
	\draw[thick, fill, radius=0.06] (1,-1) circle;
	\draw[thick, fill, radius=0.06] (2,-2) circle node[right] {$x$};
	\draw[arc] (.5,-.5) -- (-.5,-.5);
	\draw[arc] (1,-1) -- (-1,-1);
	\draw[arc] (2,-2) -- (-2,-2);
	
	\draw[thick, fill, radius=0.06] (5.5,-5.5) circle node[right] {$u_0$};
	\draw[edge,thin] (5.5,-5.5) -- (4.5,-6.5);
	\draw[edge,thin] (4.5,-6.5) -- (6.5,-6.5);
	\draw[edge,thin] (5.5,-5.5) -- (6.5,-6.5);
	
	\draw[thick, fill, radius=0.06] (4,-4) circle node[right] {$u_{1}$};
	\draw[edge] (4,-4) -- (3.5,-4.5);
	\draw[thick, fill, radius=0.06] (3.5,-4.5) circle node[left] {$v_1$};
    \draw[edge,thin] (3.5,-4.5) -- (2.5,-5.5);
	\draw[edge,thin] (2.5,-5.5) -- (4.5,-5.5);
	\draw[edge,thin] (3.5,-4.5) -- (4.5,-5.5);
	
	\draw[thick, fill, radius=0.06] (2.5,-2.5) circle node[right] {$u_{m-1}$};
	\draw[edge] (2.5,-2.5) -- (2,-3);
	\draw[thick, fill, radius=0.06] (2,-3) circle node[left] {$v_{m-1}$};
	\draw[edge,thin] (2,-3) -- (1,-4);
	\draw[edge,thin] (1,-4) -- (3,-4);
	\draw[edge,thin] (2,-3) -- (3,-4);
	\end{scope}
	\end{tikzpicture}
    \caption{\leonew{The network used in the proof of Lemma~\ref{lem:PendantTrees} and Theorem~\ref{the:orchardConnected}, in which the triangles below $u_0,v_1,\ldots ,v_{m-1}$ indicate trees rooted at those nodes.}
    }
    \label{fig:canonicalNetworkConnectedness}
\end{figure}
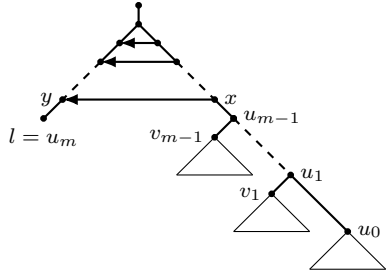

\begin{figure}
    \centering
    \begin{tikzpicture}
	 \tikzset{edge/.style={thick}}
     \tikzset{arc/.style={-{Latex[length=2mm]},thick}}
     \tikzstyle{every node}=[font=\footnotesize]
     \begin{scope}[xshift=0cm,yshift=0cm,xscale=.5,yscale=.5]
	\draw[edge] (0,.5) -- (0,0);
	\draw[thick, fill, radius=0.06] (0,0) circle node[left] {$y$};
	\draw[edge] (2,.5) -- (2,0);
	\draw[thick, fill, radius=0.06] (2,0) circle node[right] {$x$};
	\draw[arc] (2,0) -- (0,0);
	\draw[edge] (0,0) -- (0,-2);
	\draw[thick, fill, radius=0.06] (0,-2) circle node[left] {$u_i$};
	\draw[edge] (0,-2) -- (0,-3);
	\draw[thick, fill, radius=0.06] (0,-3) circle node[below] {$u_{i+1}$};
	\draw[edge] (0,-2) -- (1,-3);
	\draw[thick, fill, radius=0.06] (1,-3) circle node[below] {$v_{i}$};
	\draw[edge] (2,0) -- (2,-1.7);
	\draw[thick, fill, radius=0.06] (2,-1.7) circle node[right] {$u_{i-1}$};
	\draw[arc] (2.3,-4) -- (8,-4);
	\draw (5,-4) node[above,scale=.8] {$(y,u_i,u_{i+1})\move{(u_i,v_i)}(x,y)$};
	\end{scope}
	 \begin{scope}[xshift=4.5cm,yshift=0cm,xscale=.5,yscale=.5]
	\draw[edge] (0,.5) -- (0,-.3);
	\draw[thick, fill, radius=0.06] (0,-.3) circle node[left] {$y$};
	\draw[edge] (2,.5) -- (2,0);
	\draw[thick, fill, radius=0.06] (2,0) circle node[right] {$x$};
	\draw[arc] (1,-.3) -- (0,-.3);
	\draw[edge] (0,0) -- (0,-2);
	\draw[edge] (2,0) -- (1,-.3);
	\draw[thick, fill, radius=0.06] (1,-.3) circle node[above] {$u_i$};
	\draw[edge] (0,-2) -- (0,-3);
	\draw[thick, fill, radius=0.06] (0,-3) circle node[below] {$u_{i+1}$};
	\draw[edge] (1,-.3) -- (1,-3);
	\draw[thick, fill, radius=0.06] (1,-3) circle node[below] {$v_{i}$};
	\draw[edge] (2,0) -- (2,-1.7);
	\draw[thick, fill, radius=0.06] (2,-1.7) circle node[right] {$u_{i-1}$};
	\draw[arc] (2.3,-4) -- (8,-4);
	\draw (5,-4) node[above,scale=.8] {$(x,u_i,y)\move{(u_i,v_i)}(x,u_{i-1})$};
	\end{scope}
	\begin{scope}[xshift=9cm,yshift=0cm,xscale=.5,yscale=.5]
	\draw[edge] (0,.5) -- (0,-.3);
	\draw[thick, fill, radius=0.06] (0,0) circle node[left] {$y$};
	\draw[edge] (2,.5) -- (2,0);
	\draw[thick, fill, radius=0.06] (2,0) circle node[right] {$x$};
	\draw[arc] (2,0) -- (0,0);
	\draw[edge] (0,0) -- (0,-2);
	\draw[thick, fill, radius=0.06] (2,-1.4) circle node[right] {$u_i$};
	\draw[edge] (0,-2) -- (0,-3);
	\draw[thick, fill, radius=0.06] (0,-3) circle node[below] {$u_{i+1}$};
	\draw[edge] (2,-1.4) -- (1,-3);
	\draw[thick, fill, radius=0.06] (1,-3) circle node[below] {$v_{i}$};
	\draw[edge] (2,0) -- (2,-1.7);
	\draw[thick, fill, radius=0.06] (2,-1.7) circle node[right] {$u_{i-1}$};
	\end{scope}
	\end{tikzpicture}
    \caption{Illustration of the rNNI moves in the proof of Lemma~\ref{lem:PendantTrees}.
    }
    \label{fig:movetootherside}
\end{figure}
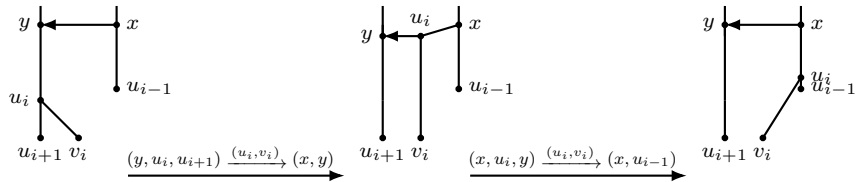

We are now ready to prove that the space of orchard networks is connected by rNNI moves, for a fixed number of reticulations and set of taxa. We basically prove this theorem by showing that each orchard network can be transformed into one particular network (see Figure~\ref{fig:canonicalNetworkConnectedness}).
This shows that, in the space of orchard networks with the same number of leaves and reticulations, there is a path from each node to one given node, and, thus, the space is connected.

\begin{theorem}\label{the:orchardConnected}
The space of binary orchard networks with $n$ leaves and $k$ reticulations is connected under rNNI moves with diameter at most \leonew{$4kn+n\ceil{\log_2(n)}+2k+6n-8$ which is~$O(kn+n\log(n))$}. 
\end{theorem}
\begin{proof}
Consider any two binary orchard networks~$N_1,N_2$ with~$k$ reticulations on the same set of~$n$ taxa. \markj{Fix an arbitrary leaf $l$.} In each of the two networks, we do the following. First move all reticulations to the top using the moves in Lemma~\ref{lem:ReticToTop}. For each reticulation, this takes at most $2n$ moves, so at most $2kn$ rNNI moves in total. Using at most $k$ rNNI moves, the horizontal arcs at the top can be reoriented such that the network has~$k$ reticulations neatly at the top, \markj{and such that $l$ is below the \leonew{head} of the lowest horizontal arc \leonew{at the top},}
by Lemma~\ref{lem:ReticsNeatlyTop}.
Using at most $2n-4$ moves, we can move all leaves except~$l$ below the tail~$x$ of the lowest horizontal arc at the top, by Lemma~\ref{lem:PendantTrees}. 
This is done on both~$N_1$ and~$N_2$, so the contribution towards the number of moves up until this point is~$2(2kn+k+2n-4)$.

Let~$\diam_{\rNNI}(n,k)$ denote the diameter of~$\Orch(n,k)$. Let~$T_i$ denote the subtree rooted at the child of~$x$ \leonew{other than~$y$ in~$N_i$ (i.e. the subtree rooted at~$u_{m-1}$ in Figure~\ref{fig:canonicalNetworkConnectedness}).} Note that~$T_i$ contains all leaves except~$l$. We can change~$T_1$ into~$T_2$ using at most $\diam_{\rNNI}(n-1,0)\leq 2n+n\ceil{\log_2(n)}$ moves \cite{li1996nearest,erdHos2021rooted}.
It follows that
\leonew{$\diam_{\rNNI}(n,k)\leq 2(2kn+k+2n-4)+\diam_{\rNNI}(n-1,0)\leq 4kn+n\ceil{\log_2(n)}+2k+6n-8=O(kn+n\log(n))$ moves.}
\end{proof}

\section{Discussion}
In this paper, we have shown that binary orchard networks can be characterized as networks with an HGT-consistent labelling, meaning that they can be obtained from a tree by inserting horizontal arcs. Hence, \leonew{this} class of networks, which was introduced for its computational benefits, also has a biological interpretation.

\leonew{This does not mean that orchard networks can only be applied in situations where reticulations represent HGT events. Although orchard networks can be drawn as HGT networks, they can also be drawn differently. Hence, orchard networks may still be useful in applications where reticulations represent hybridizations or other reticulate events. Restricting to orchard networks may exclude some scenarios in that case, but the nice mathematical properties may outweigh that, especially when they can be exploited to develop efficient algorithms.}

We have also shown that non-binary orchard networks can be characterized using this time labelling \leonew{on}
a binary refinement. \leonew{However, this characterization is less satisfying as it does not specify which binary refinement to use. We leave it as an open question to find a characterization that uses a time-labelling directly on the nonbinary network.}


To show the mathematical utility of the new characterization, we have used this new characterization to prove that the space of orchard networks is connected under rNNI moves. As mentioned, this may prove important, because some statistical network generators introduce reticulations as HGT events \cite{pons2019generation}, which naturally leads to orchard networks. Hence, if such generators are used as a prior in a Bayesian method for network inference, the prior probability of all non-orchard networks will be zero, so it is important to know that the space of orchard networks is connected.

To see whether it makes sense computationally to restrict to orchard networks, it would also be interesting to know whether our upper bound on the diameter of the space of orchard networks of order $O(kn+n\log n)$ is asymptotically tight. 
\yuki{Indeed, smaller diameters can be favourable when deciding on a network space to search through, as it could mean a shorter mixing time for Markov Chain Monte Carlo methods~\cite{klawitter2020spaces}.}
The bound is, in any case, close to the bound for tree-based networks ($O(kn+k^2+n\log n)$ \cite{erdHos2021rooted}). It is clear the asymptotic bound cannot get smaller than $O(n\log n)$, as this is tight for trees, but it may be possible to remove or reduce the $kn$ part.

Note that our results are only given for rNNI moves, and not for local tail or local head moves separately, as is done for tree-based networks in \cite{erdHos2021rooted}. It might be true that the space of orchard networks is also connected under distance-1 tail moves, or distance-2 head moves, but our proof \leonew{does not
imply} this, as we have used a mix of distance-1 tail moves and distance-1 head moves. 
It would be of interest to investigate this further.

\bibliographystyle{alpha}
\bibliography{bibliography}

\end{document}